\documentclass[12pt]{amsart}
\usepackage{amsmath,amsthm,bm}
\usepackage[psamsfonts]{amssymb}
\usepackage{type1cm}
\theoremstyle{plain}
\newtheorem{theorem}{Theorem}
\newtheorem{corollary}{Corollary}[theorem]
\newtheorem{lemma}{Lemma}
\newtheorem{proposition}{Proposition}
\newtheorem{problem}{Problem}
\theoremstyle{remark}
\newtheorem{remark}{Remark}

\numberwithin{equation}{section}

\def\al{\alpha}
\def\be{\beta}

\def\vGa{\varGamma}
\def\ga{\gamma}

\def\de{\delta}

\def\vep{\varepsilon}
\def\ze{\zeta}
\def\et{\eta}

\def\ka{\kappa}
\def\la{\lambda}

\def\rh{\rho}

\def\si{\sigma}

\def\ta{\tau}

\def\ph{\phi}

\def\ch{\chi}

\def\ps{\psi}

\def\fr{\frac}
\def\del{\partial}

\def\re{\operatorname{Re}}
\def\im{\operatorname{Im}}
\def\sgn{\operatorname{sgn}}
\def\bbZ{\mathbb{Z}}
\def\Z2{\mathbb{Z}^2}

\def\cC{\mathcal{C}}
\def\LM{\mathcal{L}\mathcal{M}}
\def\RL{\mathcal{R}\mathcal{L}}

\def\cL{\mathcal{L}}
\def\cM{\mathcal{M}}
\def\dps{\displaystyle}

\def\sump{\mathop{\sideset{}{'}\sum}}

\def\fs{\mathfrak{s}}

\def\bg{\begin}
\def\ed{\end}
\def\lf{\lfloor}
\def\rf{\rfloor}
\def\vG#1#2{\varGamma\Bigl(\begin{matrix}#1\\ #2\end{matrix}\Bigr)}
\def\tvG#1#2{\textstyle\varGamma(\begin{smallmatrix}#1\\ #2\end{smallmatrix})}

\def\1F1#1#2#3{{}_1F_1\Bigl(\begin{matrix}#1\\ #2\end{matrix};#3\Bigr)}
\def\t1F1#1#2#3{{}_1F_1(\begin{smallmatrix}#1\\ #2\end{smallmatrix};#3)}
\def\F#1#2#3{{}_2F_1\Bigl(\begin{matrix}#1\\ #2\end{matrix};#3\Bigr)}
\def\tF#1#2#3{{}_2F_1(\begin{smallmatrix}#1\\ #2\end{smallmatrix};#3)}
\def\2F3#1#2#3{{}_2F_3\Bigl(\begin{matrix}#1\\ #2\end{matrix};#3\Bigr)}
\def\I#1#2{\mathcal{I}^{\kern0.12em #1}_{\infty,#2}}
\renewcommand{\subjclass}[1]{\footnote[0]{2020 \textit{Mathematics Subject Classification.} #1.}}
\begin{document}
\title[Asymptotic expansions for Lerch zeta-functions]
{Asymptotic expansions for the Laplace-Mellin and Riemann-Liouville 
transforms of Lerch zeta-functions} 
\author[Masanori Katsurada]{Masanori Katsurada}
\address{\footnotesize Department of Mathematics, Faculty of Economics, 
Keio University, 4--1--1 Hiyoshi, Kouhoku-ku, Yokohama 223--8521, Japan}
\curraddr{\footnotesize Mathematics and Informatics Program, Faculty of Science, Kagoshima University, K{\^o}rimoto 1--21--35, Kagoshima 890--0065, Japan
}
\email{\tt katsurad@z3.keio.jp}

\thanks{The present research was supported in part by Grant-in-Aid for Scientific Research 
(No.~26400021) from JSPS}
%%
%%%%%%%%%%%%%%%%%%%%%%%%%%%%%%%%%%%%%%%%%%%%%%%%%%%%%%%%%%%%%%%%%%%%%%%%%%%%%
%%%%%%%%%%%%%%%%%%%%%%%%%%%%%%%%%%%%%%%%%%%%%%%%%%%%%%%%%%%%%%%%%%%%%%%%%%%%%
%%
\begin{abstract} 
Let $\ph(s,a,\la)$ denote the Lerch zeta-function, $\ph^{\ast}(s,a,\la)$ 
a slight modification of $\ph(s,a,\la)$ defined by extracting the (possible) 
singularity of $\ph(s,a,\la)$ at $s=1$, and $(\ph^{\ast})^{(m)}(s,a,\la)$ 
for any $m\in\bbZ$ the $m$th derivative with respect to $s$ if $m\geq0$, 
while if $m\leq0$ the $|m|$-th primitive defined with its initial point at 
$s+\infty$. The present paper aims to study asymptotic aspects of 
$(\ph^{\ast})^{(m)}(s,a,\la)$, transformed through the Laplace-Mellin 
and Riemann-Liouville operators (say, $\LM_{z;\ta}^{\al}$ and 
$\RL_{z;\ta}^{\al,\be}$, respectively) in terms of the variable $s$. 
We shall show that complete asymptotic expansions exist if $a>1$ for 
$\LM_{z;\ta}^{\al}(\ph^{\ast})^{(m)}(s+\ta,a,\la)$ and 
$\RL_{z;\ta}^{\al,\be}(\ph^{\ast})^{(m)}(s+\ta,a,\la)$ (Theorems~1--4), 
as well as for their iterated variants (Theorems~5--10), when 
the `pivotal' parameter $z$ (of the transforms) tends to both $0$ and 
$\infty$ through appropriate sectors. Most of our results include  
any vertical ray in their region of validity; this allows us to deduce 
complete asymptotic expansions along vertical lines $(s,z)=(\si,it)$ as 
$t\to\pm\infty$ (Corollaries~2.1,~4.1,~6.1 and~8.1). 
\end{abstract} 
%%%%%%%%%%%%%%%%%%%%%%%%%%%%%%%%%%%%%%%%%%%%%%%%%%%%%% 
\maketitle
%%%%%%%%%%%%%%%%%%%%%%%%%%%%%%%%%%%%%%%%%%%%%%%%%%%%%% 
%%
\section{Introduction}
Throughout the paper, $s=\si+it$ is a complex variable, and $z=x+iy$ a complex parameter 
(with real coordinates $\si$, $t$, $x$ and $y$), $a$ and $\la$ are real parameters 
with $a>0$, and the notation $e(s)=e^{2\pi is}$ is frequently used. The Lerch zeta-function 
$\ph(s,a,\la)$ is defined by the Dirichlet series 
\begin{align*}
\ph(s,a,\la)
=\sum_{l=0}^{\infty}e(\la l)(a+l)^{-s}
\qquad
(\si>1),
\tag{1.1}
\end{align*}
and its meromorphic continuation over the whole $s$-plane 
(cf.~\cite{lerch1887}\cite{lipschitz1889}); this reduces if $\la\in\mathbb{Z}$ to the 
Hurwitz zeta-function $\ze(s,a)$, and further to the Riemann zeta-function $\ze(s)=\ze(s,1)$. 
Note that the domain of the parameter $a$ may be extended to the whole sector $|\arg z|<\pi$ 
through the procedure in \cite{katsurada1998}. 
 
Let $\vGa(s)$ denote the gamma function, $\al$ and $\be$ be complex numbers with positive real parts, $f(z)$ a function holomorphic in the sector $|\arg z|<\pi$, and write $X_+
=\max(0,X)$ for any $X\in\mathbb{R}$. We introduce here the Laplace-Mellin and 
Riemann-Liouville (or Erd{\'e}lyi-K{\"o}ber) transforms of $f(z)$, given by  
\begin{align*}
\LM_{z;\ta}^{\al}f(\ta)
&=\fr1{\vGa(\al)}\int_0^{\infty}f(z\ta)\ta^{\al-1}e^{-\ta}d\ta,
\tag{1.2}\\
\RL_{z;\ta}^{\al,\be}f(\ta)
&=\fr{\vGa(\al+\be)}{\vGa(\al)\vGa(\be)}
\int_0^{\infty}f(z\ta)\ta^{\al-1}(1-\ta)_+^{\be-1}d\ta
\tag{1.3}
\end{align*}
with the normalization gamma multiples, provided that the integrals converge; the factor 
$\ta^{\al-1}$ is inserted to secure the convergence as $\ta\to0^+$, while $e^{-\ta}$ 
and $(1-\ta)_+^{\be-1}$ have effects to extract the portions of $f(\ta)$ corresponding 
to $\ta=O(z)$. Let $\de_{\mathbb{Z}}(\la)$ denote the symbol which equals $1$ or $0$ 
according to $\la\in\mathbb{Z}$ or otherwise, and set 
\begin{align*}
\ps(s,z)=\fr{z^{1-s}}{s-1}.
\tag{1.4}
\end{align*}
We further introduce a slignt modification $\ph^{\ast}(s,z,\la)$ of $\ph(s,z,\la)$,  
defined by 
\begin{align*}
\ph^{\ast}(s,z,\la)=\ph(s,z,\la)-\de_{\mathbb{Z}}(\la)\ps(s,z)
=\begin{cases}
{\dps\ze(s,z)-\ps(s,z)}&\quad\text{if $\la\in\mathbb{Z}$},\\
{\dps\ph(s,z,\la)}&\quad\text{otherwise},
\end{cases}
\tag{1.5}
\end{align*} 
which removes the only (possible) singularity at $s=1$, where the (consistent)  
notation $\ze^{\ast}(s,z)=\ph^{\ast}(s,z,\la)$ if $\la\in\mathbb{Z}$ is used throughout. 
The exclusion here in (1.5) has an advantage over enlarging satisfactorily the region of 
validity (so as to include vertical directions) of our asymptotic expansions. 
Next let $f^{(m)}(s)$, for any entire function $f(s)$, denote its $m$th derivative if 
$m=0,1,2,\ldots$, while its $|m|$th primitive if $m=-1,-2,\ldots$, defined inductively by
\begin{align*}
f^{(m)}(s)=\int_{s+\infty}^{s}f^{(m+1)}(w)dw=-\int_0^{0+\infty}f^{(m+1)}(s+u)du
\qquad(m=-1,-2,\ldots),\tag{1.6}
\end{align*} 
subject to convergence, where the path of integration is the horizontal ray. 

The principal aim of the present paper is to study asymptotic aspects of the Laplace-Mellin 
and Riemann-Liouville transforms of the modified Lerch zeta-function, defined by  
\begin{gather*}
\LM_{z;\ta}^{\al}(\ph^{\ast})^{(m)}(s+\ta,a,\la)
=\fr1{\vGa(\al)}\int_0^{\infty}
(\ph^{\ast})^{(m)}(s+z\ta,a,\la)\ta^{\al-1}e^{-\ta}d\ta
\tag{1.7}\\
\RL_{z;\ta}^{\al,\be}(\ph^{\ast})^{(m)}(s+\ta,a,\la)
=\fr{\vGa(\al+\be)}{\vGa(\al)\vGa(\be)}
\int_0^{\infty}(\ph^{\ast})^{(m)}(s+z\ta,a,\la)\ta^{\al-1}(1-\ta)_+^{\be-1}d\ta
\tag{1.8}
\end{gather*}
for any $m\in\mathbb{Z}$, where the conditions $a>1$ and $|\arg z|\leq\pi/2$ are 
required in (1.7) for convergence. We shall show that complete asymptotic expansions 
exist if $a>1$ for (1.7) and (1.8) (Theorems~1--4) when the `pivotal' parameter $z$ 
of the transforms tends to both $0$ and $\infty$ through appropriate sectors. 
Moreover the iterations of (1.2) and (1.3) lead us to generate various transformations, 
some of whose primary situations are treated to establish complete asymptotic 
expansions for (1.5) transformed through such operators (Theorems~5--10). 
Most of our results include any vertical ray within their region of validity; 
this allows us to deduce complete asymptotic expansions along vertical lines 
with $(s,z)=(\si,it)$ as $t\to\pm\infty$ (Corollaries~2.1, 4.1, 6.1 and 8.1). 
It seems that the zeta-functions transformed through, e.g. (1.2), (1.3), (3.2), 
(3.9) and (3.16) are potentialy rich objects, however, have been fairly rare 
subjects of research in the literature.  

As for the methods used, crucial roles in the proofs are played by the Mellin-Barnes type 
integrals in (6.4), (6.11), (7.3), (7.13), (7.19) and (7.29) below; a key ingredient in manupilating these 
integrals is the vertical estimate (4.12) for the auxiliary zeta-function defined by (2.2) with (2.1).

We next give an overview of history of research related to asymptotic 
aspects of the integral transforms of zeta-functions. The study of Laplace transforms for 
the mean square of $\ze(s)$ seems to be initiated by Hardy-Littlewood \cite{hardy-littlewood1918}, 
who obtained the asymptotic relation, say,
\begin{align*}
\cL_{1/2}(\vep)
=\int_0^{\infty}\biggl|\ze\Bigl(\fr12+it\Bigr)\biggr|^2e^{-\vep t}dt
\sim\fr{1}{\vep}\log\fr{1}{\vep},
\qquad(\vep\to0^+),
\end{align*}
in connection with the research of asymptotic behaviour of the 
(upper-truncated) mean square of $\ze(1/2+it)$, in the form $\int_0^T|\ze(1/2+it)|^2dt$  
as $T\to+\infty$. Wilton \cite{wilton1930} then refined the result above to an asymptotic 
expansion with the error term $O\{\vep^{-1/2}\log^{3/2}(1/\vep)\}$ $(\vep\to0^+)$, 
which was in fact replaced by a complete asymptotic expansion by 
K{\"o}ber~\cite{kober1936}. Atkinson \cite{atkinson1939} finally succeeded (through a rather 
elementary argument) in establishing for any integer $N\geq0$ that 
\begin{align*}
\cL_{1/2}(\vep)
&=\fr{1}{\vep}\log\fr{1}{\vep}-\fr{\log2\pi-\ga_0}{\vep}
+\sum_{n=0}^{N-1}\Bigl(a_n+b_n\log\fr{1}{\vep}\Bigr)\vep^n
+O\Bigl(\vep^N\log\fr{1}{\vep}\Bigr)
\end{align*}
as $\vep\to0^+$ with some constants $a_n$, $b_n$ $(n=0,1,\ldots; b_0=0)$ and the $0$th 
Euler-Stieltjes constant $\ga_0=-\vGa'(1)$ (cf.~\cite[p.34,1.12(17)]{erdelyi1953a}). 

It can be said that recent major developments into this direction have 
been made, to a greater or less extent, in the spirit of Atkinson's influential 
work~\cite{atkinson1949} on the error term $E(T)$ of the upper-truncated mean square of 
$\ze(1/2+it)$, where the innovative treatment of the product $\ze(u)\ze(v)$ with 
independent complex variables $(u,v)$ was made toward the evantual application upon 
taking $u=1/2+it$ and $v=1/2-it$. A more general weighted mean square 
$\cL_{\rh}(s)=\int_0^{\infty}|\ze(\rh+ix)|^2e^{-sx}dx$ was treated in the late 
1990's by Jutila~\cite{jutila1998}, who made a detailed study of $\cL_{\rh}(s)$, especially 
on the critical line $\rh=1/2$, while obtaining its asymptotic formula as $s\to0$ through 
the sector $|\arg s|<\pi/2$, and also applied it to rederive the classical (so-called) Atkinson's 
formula for $E(T)$. Further study of $\cL_{\rh}(s)$ has been carried out by 
Ka{\v c}inskait{\. e}-Laurin{\v c}ikas~\cite{kacinskaite-laurincikas2009}. On the other 
hand, the (lower-truncated) Mellin transform $\cM_{k,\rh}(s)=\int_1^{\infty}|\ze(\rh+ix)|^{2k}x^{-s}dx$ 
$(k=1,2,\ldots)$ was explored by Ivi{\'c}-Jutila-Motohashi~\cite{ivic-jutila-motohashi2002}, who 
applied their results to investigate the higher power moments, in particular the eighth power 
moment, of $\ze(s)$. Research subsequent to \cite{ivic-jutila-motohashi2002} was due to 
Ivi{\'c}~\cite{ivic2005}, while Laurin{\v c}ikas \cite{laurincikas2011} made a detailed 
study of the case $k=1$, i.e.~the mean square case, of $\cM_{k,\rh}(s)$. 

As for the relevant asymptotic aspects of allied zeta-functions, the Laplace transform 
$\int_0^{\infty}|L(1/2+it,\ch)|^2e^{-\vep t}dt$, where $L(s,\ch)$ denotes the $L$-function 
attached to a primitive Dirichlet character $\ch$ modulo $q\geq2$, was treated in the 
same paper above by K{\"o}ber~\cite{kober1936}. Next let $a$, $b$, $\mu$ and $\nu$ 
be arbitrary real parameters, and $\ps_{\bbZ^2}(s;a,b;\mu,\nu;z)$ denote the 
generalized Epstein zeta-function, defined for $\im z>0$ by 
\begin{align*}
\ps_{\bbZ^2}(s;a,b;\mu,\nu;z)
&=\sump_{m,n=-\infty}^{\infty}\fr{e\{(a+m)\mu+(b+n)\nu\}}{|a+m+(b+n)z|^{2s}}
\qquad(\si>1),
\tag{1.9}
\end{align*}
and its meromorphic continuation over the whole $s$-plane, where the impossible term 
$1/0^{2s}$ is to be excluded; the particular case $(a,b)\in\bbZ^2$ and $(\mu,\nu)=(0,0)$ 
reduces to the classical Epstein zeta-function $\ze_{\bbZ^2}(s;z)$. The author~\cite{katsurada2007} 
has shown that complete asymptotic expansions exist for $\ze_{\mathbb{Z}^2}(s;z)$ when 
$y=\im z\to+\infty$, and also for the Laplace-Mellin transform $\LM_{Y;y}^{\al}\ze_{\bbZ^2}(s;x+iy)$ 
when $Y\to+\infty$. The method developed in \cite{katsurada2007} could be extended in 
\cite{katsurada2015} to show that similar expansions exist further for (1.9) when 
$y\to+\infty$, as well as for the Riemann-Liouville transform 
$\RL_{Y;y}^{\al,\be}\ze_{\bbZ^2}(s;x+iy)$ when $Y\to+\infty$.   
%%%%%%%%%%%%%%%%%%%%%%%%%%%%%%%%%%%%%%%%%%%%%%%%%%%%%%%%%%%%%%%%%%%%%%%%%%%%%%%%%%%%%%%%%%%%%%%%%

The paper is organized as follows. Various complete asymptotic expansions for the 
transforms (1.7) and (1.8) are presented in the next section, and those for their 
iterated variants in Section~3, together with their applications. Several results 
necessary for our proofs are prepared in Section~4, while Sections~5--6 and 7
are devoted to establishing Theorems~1--4 and Theorems~5--10 respectively. 
%%%%%%%%%%%%%%%%%%%%%%%%%%%%%%%%%%%%%%%%%%%%%%%%%%%%%%%%%%%%%%%%%%%%%%%%%%%%%%%%%%%%%%%%%%%%%%%%%
\section{Statement of results (1)}
%%%%%%%%%%%%%%%%%%%%%%%%%%%%%%%%%%%%%%%%%%%%%%%%%%%%%%%%%%%%%%%%%%%%%%%%%%%%%%%%%%%%%%%%%%%%%%%%%%%%%

We first introduce the Hadamard type operators with the initial point at $\infty$, 
defined for any $(r,s)\in\mathbb{C}^2$ by 
\begin{align*}
\I{r}{s}f(s)=\fr1{\vGa(r)\{e(r)-1\}}\int_{\infty}^{(0+)}f(s+z)z^{r-1}dz\tag{2.1}
\end{align*}
if $f(s+x)$ belongs (as a function of $x$) to the class $x^{1-\re r}L^1_x[0,+\infty[$. 
Here the path of integration is a contour which starts from $\infty$, proceeds along 
the real axis to a small $\et>0$, encircles the origin counter-clockwise, and returns 
to $\infty$ along the real axis; $\arg z$ varies from $0$ to $2\pi$ along the contour. 
Then the auxiliary zeta-function $\ph^{\ast}_r(s,a,\la)$ is defined, for any 
$(r,s)\in\mathbb{C}^2$ and for any $(a,\la)\in\mathbb{R}^2$ with $a>1$, by 
\begin{align*}
\ph^{\ast}_{r}(s,a,\la)=\I{r}{s}\ph^{\ast}(s,a,\la),\tag{2.2}
\end{align*}
which is crucial in describing our results, and also of some interest in itself, since the relation 
\begin{align*}
\ga_m(a,\la)
&=\fr{(-1)^m}{m!}\{\ph_{-m}^{\ast}(1,a,\la)+\log^ma\}
\qquad(m=0,1,\ldots)
\end{align*}
holds if $a>1$, where $\ga_m(a,\la)$ are the extensions of the generalized Euler-Stieltjes constants (cf.~\cite[p.41, 1.8,(1.123)]{ivic1985}) to Lerch zeta-functions. Further, let $(s)_n=\vGa(s+n)/\vGa(s)$ for any $n\in\mathbb{Z}$ denote the rising factorial 
of $s$, and write  
\begin{align*}
\vG{\al_1,\ldots,\al_m}{\be_1,\ldots,\be_n}
=\fr{\prod_{h=1}^m\vGa(\al_h)}{\prod_{k=1}^n\vGa(\be_k)}
\end{align*}
for complex numbers $\al_h$ and $\be_k$ $(h=1,\ldots,m$; $k=1,\ldots,n)$.

We now state our results on the Laplace-Mellin transform (1.7). The following 
Theorems~1~and~2 assert the asymptotic expansions as $z\to0$ and as 
$z\to\infty$ respectively. 
\begin{theorem}%%%%Theorem 1%%%%%%%%%%%%%%%%%%%%%%%%%%%
Let $\al$ be any complex with positive real part, $s$ any complex variable, $a$ and $\la$ any real parameters with $a>1$, and $m$ any integer. Then for any integer $N\geq0$ we have  
\begin{align*}
\LM_{z;\ta}^{\al}(\ph^{\ast})^{(m)}(s+\ta,a,\la)
&=(-1)^m\sum_{n=0}^{N-1}\fr{(-1)^{n}(\al)_n}{n!}\ph^{\ast}_{-n-m}(s,a,\la)z^{n}
\tag{2.3}\\
&\quad+R_{m,N}^{1,+}(s,a,\la;z)
\end{align*}
in the sector $|\arg z|<\pi$. Here the reminder $R_{m,N}^{1,+}(s,a,\la;z)$ is expressed by the Mellin-Barnes type integral in (5.7) below, and satisfies the estimate
\begin{align*}
R_{m,N}^{1,+}(s,a,\la;z)=O\{(|t|+1)^{\max(0,\lf 2-\si\rf)}|z|^{N}\}
\tag{2.4}
\end{align*}
as $z\to0$ through $|\arg z|\leq\pi-\et$ with any small $\et>0$, where the implied 
$O$-constant depends at most on $\si$, $a$, $\la$, $\al$, $m$, $N$ and $\et$.
\end{theorem}
%%%%End of Theorem~1%%%%%%%%%%%%%%%%%%
%%%%%%%%%%%%%%%%%%%%%%%%%%%%%%%%%%%%%%%%%%%%%%%%%%%%%%%%%%%%%%%%%%%%%%%%%%%%%%%%%%%%%%%%%%%%%%%%%%%%Theorem~2%%%%%%%%%%%%%
\begin{theorem}%%Theorem~2%%%%%%%%%%%%%%
Upon the same settings as in Theorem~1 we have  
\begin{align*}
\LM_{z;\ta}^{\al}(\ph^{\ast})^{(m)}(s+\ta,a,\la)
&=(-1)^m\sum_{n=0}^{N-1}\fr{(-1)^{n}(\al)_n}{n!}\ph^{\ast}_{\al+n-m}(s,a,\la)z^{-\al-n}
\tag{2.5}\\
&\quad+R^{1,-}_{m,N}(s,a,\la;z)
\end{align*}
in the sector $|\arg z|<\pi$. Here the reminder $R^{1,-}_{m,N}(s,a,\la;z)$ is expressed by the Mellin-Barnes type integral (5.10) below, and satisfies the estimate
\begin{align*}
R^{1,-}_{m,N}(s,a,\la;z)
=O\{(|t|+1)^{\max(0,\lf 2-\si\rf)}|z|^{-\re\al-N}\}
\tag{2.6}
\end{align*}
as $z\to\infty$ through $|\arg z|\leq\pi-\et$ with any small $\et>0$, where the implied $O$-constant depennds at most on $\si$, $a$, $\la$, $\al$, $m$, $N$ and $\et$. 
\end{theorem}
%%%%%%%%%%%%%%%%%%%%%%%%%%%%%%%%%%%%%%%%%%%%%%%%%%%%%%%%%%%%%%%%%%%%%%%%%%%%%%%%%%%%%%%%%%%%%%%%%%
We write $\sgn X=X/|X|$ for any real $X\neq0$. The case $(s,z)=(\si,it)$ with $\si,t\in\mathbb{R}$ of Theorem~2 asserts the following asymptotic expansion along vertical lines.  
%%%%%%%%%%%%%%%%%%%%%%%%%%%%%%%%%%%%%%%%%%%%%%%%%%%%%%%%%%%%%%%%%%%%%%
%%%Corolary~2.1%%%%%%%%%%%%%%%%%%%%%%%%%%%%%%%%%%%%%%%%%%%%%%%%%%%%%%%%%%%%%%%%%%%
\begin{corollary}%%Corollary~2.1%%%%%%%%%%%%%%%%%%%%%%%%%%%
Upon the same settings as in Theorem~2 with any real $\si$, we have
\begin{align*}
&\LM_{t;\ta}^{\al}(\ph^{\ast})^{(m)}(\si+it,a,\la)\tag{2.7}\\
&\quad=(-1)^m\sum_{n=0}^{N-1}\fr{(-1)^n(\al)_n}{n!}
\ph^{\ast}_{\al+n-m}(\si,a,\la)(e^{(\sgn t)\pi i/2}|t|)^{-\al-n}+O(|t|^{-\re\al-N})
\end{align*}
as $t\to\pm\infty$, where the implied $O$-constant depends at most on $\si$, $a$, $\la$, $\al$, $m$ and $N$.
\end{corollary}%%%End of Corollary~2.1
%%%%%%%%%%%%%%%%%%%%%%%%%%%%%%%%%%%%%%%%%%%%%%%%%
%%%%%%%%%%%%%%%%%%%%%%%%%%%%%%%%%%%%%%%%%%%%%%%%%%%
We proceed to state our results on the Riemann-Liouville transform (1.8). 
The following Theorems~3~and~4 assert the asymptotic 
expansions as $z\to0$ and as $z\to\infty$ respectively.   
%%%%%%%%%%%Theorem~3
\begin{theorem}%%Theorem~3%%%%%%%%%%%%%%%%%%%%%%%%%%%%%5
Upon the same settings as in Theorem~1 with any complex $\be$ of positive real part, we have 
\begin{align*}
\RL_{z;\ta}^{\al,\be}(\ph^{\ast})^{(m)}(s+\ta,a,\la)
&=(-1)^m\sum_{n=0}^{N-1}\fr{(-1)^{n}(\al)_n}{(\al+\be)_nn!}
\ph^{\ast}_{-n-m}(s,a,\la)z^n\tag{2.8}\\
&\quad+R_{m,N}^{2,+}(s,a,\la;z)
\end{align*}
in the sector $|\arg z|<\pi$. Here the reminder $R^{2,+}_{m,N}(s,a,\la;z)$ is expressed by the Mellin-Barnes type integral in (6.6) below, and satisfies the estimate
\begin{align*}
R^{2,+}_{m,N}(s,a,\la;z)
=O\{(|t|+1)^{\max(0,\lf 2-\si\rf)}|z|^N\}
\tag{2.9}
\end{align*} 
as $z\to0$ through $|\arg z|\leq\pi/2-\et$ with any small $\et>0$, where the implied 
$O$-constant depends at most on $\si$, $a$, $\la$, $\al$, $\be$, $m$, $N$ and $\et$. 
\end{theorem}
%%%%%%%%%%%%%%%%%%%%%%%%%%%%%%%%%%%%%%%%%%%%%%%%%%%%%%%%%%%%%%%%%%%%%%%%%%%%%%%%%%%%%%%%%%%%%%%%%%%%%

We hereafter write $\vep(z)=\sgn(\arg z)$ for any complex $z$ in the sectors $|\arg z|>0$, and $\lf x\rf$ for $x\in\mathbb{R}$ the greatest integer not exceeding $x$.
%%%%%%%%%%%%%%%%%%%%%%%%%%%%%%Theorem 4%%%%%%%%%%%%%%%%%%%%%%%%%%%%%%%%%%%%%%%%%%%%%%%%%%%%%%%%%%%%%%
\begin{theorem}%%%%%Theorem~4
Upon the same setting as in Theorem~3, for any integers $N_j$ $(j=1,2)$ with $N_1\geq\lf\re\be\rf$ and $N_2\geq\lf\re\al\rf$ we have  
\begin{align*}
&\RL_{z;\ta}^{\al,\be}(\ph^{\ast})^{(m)}(s+\ta,a,\la)\tag{2.10}\\ 
&\quad=(-1)^m\vG{\al+\be}{\be}e^{-\vep(z)\pi i\al}
\Biggl\{\sum_{n=0}^{N_1-1}\fr{(-1)^{n}(\al)_n(1-\be)_n}{n!}\\
&\qquad\times\ph^{\ast}_{\al+n-m}(s,a,\la)(e^{-\vep(z)\pi i}z)^{-\al-n}
+R^{2,-}_{1,m,N_1}(s,a,\la;z)\Biggr\}\\
&\qquad+(-1)^m\vG{\al+\be}{\al}e^{\vep(z)\pi i\be}
\Biggl\{\sum_{n=0}^{N_2-1}\fr{(-1)^{n}(\be)_n(1-\al)_n}{n!}\\
&\qquad\times\ph^{\ast}_{\be+n-m}(s+z,a,\la)z^{-\be-n}
+R^{2,-}_{2,m,N_2}(s,a,\la;z)\Biggr\}
\end{align*} 
in the sectors $0<|\arg z|<\pi$. Here the reminders $R^{2,-}_{j,m,N_j}(s,a,\la;z)$ $(j=1,2)$ 
are expressed by the Mellin-Barnes type integrals in (6.13) below, and satisfy the estimates
\begin{align*}
R^{2,-}_{1,m,N_1}(s,a,\la;z)
&=O\{(|t|+1)^{\max(0,\lf 2-\si\rf)}|z|^{-\re\al-N_1}\},
\tag{2.11}\\
R^{2,-}_{2,m,N_2}(s,a,\la;z)
&=O\{(|t+y|+1)^{\max(0,\lf 2-\si-x)\rf)}|z|^{-\re\be-N_2}\}
\tag{2.12}
\end{align*}
both as $z\to\infty$ through $\eta\leq|\arg z|\leq\pi-\et$ with any small $\et>0$, where 
the implied $O$-constant in {\rm(2.11)} depends at most on $\si$, $a$, $\la$, $\al$, $\be$, 
$m$, $N_1$ and $\et$, while that in {\rm(2.12)} at most on $\si$, $x$, $a$, $\la$, $\al$, $\be$, 
$m$, $N_2$ and $\et$. 
\end{theorem}
\begin{remark}
Stokes' phenomenum for confluent hypergeometric functions (see (6.7) below) effects splitting the shape of the asymptotic expansion in (2.10) into two sectors $0<|\arg z|<\pi$. 
\end{remark}
The case $(s,z)=(\si,it)$ with $\si,t\in\mathbb{R}$ of Theorem~4 asserts the 
following asymptotic expansion along vertical lines.
%%%%%%Corollary~4.1%%%%%%%%%%%%
\begin{corollary}%%Corollary~4.1%%%%%%%%%%%%%%%%%%%
Upon the same settings as in Theorem~4 with any real $\si$, we have
\begin{align*}
&\RL_{t;\ta}^{\al,\be}(\ph^{\ast})^{(m)}(\si+i\ta,a,\la)\tag{2.13}\\ 
&\quad=(-1)^m\vG{\al+\be}{\be}e^{-(\sgn t)\pi i\al}
\Biggl\{\sum_{n=0}^{N_1-1}\fr{(-1)^{n}(\al)_n(1-\be)_n}{n!}\\
&\qquad\times\ph^{\ast}_{\al+n-m}(\si,a,\la)(e^{-(\sgn t)\pi i/2}|t|)^{-\al-n}
+O(|t|^{-\re\al-N_1})\Biggr\}\\
&\qquad+(-1)^m\vG{\al+\be}{\al}e^{(\sgn t)\pi i\be}
\Biggl\{\sum_{n=0}^{N_2-1}\fr{(-1)^{n}(\be)_n(1-\al)_n}{n!}\\
&\qquad\times\ph^{\ast}_{\be+n-m}(\si+it,a,\la)(e^{(\sgn t)\pi i/2}|t|)^{-\be-n}
+O(|t|^{\max(0,\lf2-\si\rf)-\re\be-N_2})\Biggr\}
\end{align*}
as $t\to\pm\infty$, where the implied $O$-constants depend at most on $\si$, $\al$, $\be$, $m$ and $N_j$ $(j=1,2)$.  
\end{corollary}
%%%%%%%%%%%%%%%%%%%%%%%%%%%%%%%%%%%%%%%%%%%%%%%%%%%%%%%%%%%%%%%%%%%%%%%%%%%%%%%%%%%%%%%%%%%%%%% 
\section{Statement of results (2)} 
%%%%%%%%%%%%%%%%%%%%%%%%%%%%%%%%%%%%%%%%%%%%%%%%%%%%%%%%%%%%%%%%%%%%%%%%%%%%%%%%%%%%%%%%%%%%%
In this section, we state our results on the iterated transforms. For this, let $K_{\nu}(Z)$ 
denote the modified Bessel function of the third kind, defined by   
\begin{align*}
K_{\nu}(Z)
=\fr12\int_0^{\infty}\exp\biggl\{-\fr{Z}2\Bigl(\xi+\fr{1}{\xi}\Bigr)\biggr\}
\xi^{-\nu-1}d\xi\tag{3.1}
\end{align*}
in $|\arg Z|<\pi/2$ and for any $\nu\in\mathbb{C}$ and (cf.~\cite[p.82,~7.12(23)]{erdelyi1953b}), 
where the domain of $Z$ is extended to $|\arg Z|<\pi$ by rotating appropriately the path of 
integration. We can then show the following expression.
\begin{proposition}
For any $f(\ta)$ holomorphic in $|\arg\ta|<\pi$ we have 
\begin{align*}
\LM^{\be}_{z;\ta_2}\LM^{\al}_{\ta_2;\ta_1}f(\ta_1)
=\fr{2}{\vGa(\al)\vGa(\be)}
\int_0^{\infty}f(z\ta)\ta^{(\al+\be)/2-1}K_{\al-\be}(2\sqrt{\ta})d\ta,
\tag{3.2}
\end{align*} 
provided that the integral converges, for which it suffices to choose a class of functions $f(z)$ 
such that $O(1)$ $(z\to0)$ and $O\{\exp(|z|^{\mu})\}$ with $\mu<1/2$ $(z\to\infty)$ both 
through $|\arg z|<\pi$.
\end{proposition}
%%%%%%%%%%%%%%%%%%%%%%%%%%%%%%%%%%%%%%%%%%%%%%%%%%%%%%%%%%%%%%%%%%%%%%%%%%%%%%%%%%%%%
\begin{proof}
The left side of (3.2) equals
\begin{align*}
&\fr{1}{\vGa(\al)\vGa(\be)}
\int_0^{\infty}\int_0^{\infty}f(z\ta_1\ta_2)\ta_1^{\al-1}\ta_2^{\be-1}e^{-\ta_1-\ta_2}d\ta_1d\ta_2\\
&\quad=\fr{1}{\vGa(\al)\vGa(\be)}
\int_0^{\infty}f(z\ta)\ta^{\al-1}
\int_0^{\infty}e^{-\ta_2-\ta/\ta_2}\ta_2^{\be-\al-1}d\ta_2d\ta, 
\end{align*}
where we substitute the variable $\ta_1=\ta/\ta_2$ and then interchange the order of  
the $\ta$- and $\ta_2$-integral on the first line. Here the resulting inner $\ta_2$-integral 
is further evaluated as $2\ta^{(\be-\al)/2}K_{\al-\be}(2\sqrt{\ta})$ by (3.1); this concludes 
(3.2). The (sufficient) condition for the convergence of the integral in (3.2) can be seen from 
the fact that the Bessel function in (3.2) is of order 
$\asymp\ta^{\re(\al-\be)}+\ta^{\re(\be-\al)}$ $(\ta\to0^+)$ and 
$\asymp e^{-2\sqrt{\ta}}\ta^{-1/4}$ $(\ta\to+\infty)$ (cf.~\cite[p.5, 7.2.2(13); p.24, 7.4.1(1)]{erdelyi1953b}).
\end{proof}
%%%%%%%%%%%%%%%%%%%%%%%%%%%%%%%%%%%%%%%%%%%%%%%%%%%%%%%%%%%%%%%%%%%%%%%
We proceed to state our results on the iteration of two Laplace-Mellin transforms. 
%%Theorem~5%%%%%%%%%%%%%%%%%%%%%%%%%%%%%%
\begin{theorem}%%%%%%%%%%%%%%%%%%%%%%%%%%%
Upon the same settings as in Theorem~3 we have 
\begin{align*}
&\LM_{z;\ta_2}^{\be}\LM_{\ta_2;\ta_1}^{\al}
(\ph^{\ast})^{(m)}(s+\ta_1,a,\la)\tag{3.3}\\
&\quad=(-1)^m\sum_{n=0}^{N-1}\fr{(-1)^{n}(\al)_n(\be)_n}{n!}
\ph^{\ast}_{-n-m}(s,a,\la)z^n
+R^{3,+}_{m,N}(s,a,\la;z) 
\end{align*}
in the sector $|\arg z|<3\pi/2$. Here the  reminder $R^{3,+}_{m,N}(s,a,\la;z)$ is expressed by the Mellin-Barnes 
type integral in (7.5) below, and satisfies the estimate
\begin{align*}
R^{3,+}_{m,N}(s,a,\la;z)
&=O\{(|t|+1)^{\max(0,\lf2-\si\rf)}|z|^N\}
\tag{3.4}
\end{align*}
as $z\to0$ through $|\arg z|\leq 3\pi/2-\et$ with any small $\et>0$, where the implied 
$O$-constant depends at most on $\si$, $a$, $\la$, $\al$, $\be$, $m$, $N$ and $\et$. 
\end{theorem}
%%%%%%%%%%%%%%%%%%%%%%%%%%%%%%%%%%%%%%%%%%%%%%%%%%%%%%%%%%%%%%%%%%%%%%%%%%%%%%%%%%%%%%%%%%%%%%%%%%%Theorem~6%%%%%%%%%%%%%%%%%%%%%%%%%%%%%%%%%%%%%%%%%%%%%%%%%%%%%%%%%%%%%%%
\begin{theorem}%%%%%%%%%%%%%%Theorem 6%%%%%%%%%%%%%%%%%%%%%%%%%%%%%%%%%%%%%%%
Upon the same settings as in Theorem~3 with $N'=N-\lf\re(\be-\al)\rf$, except the case $\al-\be\in\mathbb{Z}$  we have   
\begin{align*}
&\LM^{\be}_{z;\ta_2}\LM^{\al}_{\ta_2;\ta_1}
(\ph^{\ast})^{(m)}(s+\ta,a,\la)
\tag{3.5}\\
&\quad=(-1)^m\vG{\be-\al}{\be}\sum_{n=0}^{N-1}\fr{(\al)_n}{(1+\al-\be)_nn!}
\ph^{\ast}_{\al+n-m}(s,a,\la)z^{-\al-n}\\
&\qquad+(-1)^m\vG{\al-\be}{\al}\sum_{n=0}^{N'-1}\fr{(\be)_n}{(1-\al+\be)_nn!}
\ph^{\ast}_{\be+n-m}(s,a,\la)z^{-\be-n}\\
&\qquad+R^{3,-}_{m,N}(s,a,\la;z)
\end{align*}
in the sector $|\arg z|<3\pi/2$. Here the reminder $R^{3,-}_{m,N}(s,a,\la;z)$ is expressed by the Mellin-Barnes type inetgral in (7.7) below, and satisfies the estimate
\begin{align*}
R^{3,-}_{m,N}(s,a,\la;z)
&=O\{(|t|+1)^{\max(0,\lf2-\si\rf)}|z|^{-\re\al-N}\}
\tag{3.6}
\end{align*} 
as $z\to\infty$ through $|\arg z|\leq\pi-\et$ with any small $\et>0$, where the implied 
$O$-constant depends at most on $\si$, $a$, $\la$, $\al$, $\be$, $m$, $N$ and $\et$. 
%%%%%%%%%%%%%%%%%%%%%%%%%%%%%%%%%%%%%
\end{theorem}
%%%%%%%%%%%%%%%%%%%%%%%%%%%%%%%%%%%%%%%%%%%%%%%%%%%%%%%%%%%%%%%%%%%%%
\begin{remark}
The limiting form of (3.5) when $\be\to\al+l$ $(l\in\mathbb{Z})$ complements the excluded case  
$\al-\be\in\mathbb{Z}$ of Theorem~6; similar situations also occur in Theorems~8~and~10, however 
the details are to be omitted for brevity.  
\end{remark}
%%%%%%%%%%%%%%%%%%%%%%%%%%%%%%%%%%%%%%%%%%%%%%%%%%%%%%%%%%%%%%%%%%%%%%%%%%%%%%%%%%%%%%%%%%%%%%%
The case $(s,z)=(\si,it)$ with $\si,t\in\mathbb{R}$ of Theorem~6 yields the following 
asymptotic expansion along vertical lines. 
%%%%%%%%%%%%%%%%%%%%%%%%%%%%%%%%%%%%%%%%%%%%%%%%%%%%%%%%%%%%%%%%%%%%%%
%%%Corollary~6.1%%%%%%%%%%%%%%%%%%%%%%%%%%%%%%%%%%%%%%%%%%%%%%%%%%%%%%%%%%%%%%%%%%%%%
\begin{corollary}%%%%%%Corollary~6.1%%%%%%%%%%%%%%
Upon the same settings as in Theorem~6 with any real $\si$, we have 
\begin{align*}
&\LM^{\be}_{t;\ta_2}\LM^{\al}_{\ta_2;\ta_1}
(\ph^{\ast})^{(m)}(\si+\ta,a,\la)
\tag{3.7}\\
&\quad=(-1)^m\vG{\be-\al}{\be}\sum_{n=0}^{N-1}\fr{(\al)_n}{(1+\al-\be)_nn!}
\ph^{\ast}_{\al+n-m}(\si,a,\la)(e^{(\sgn t)\pi i/2}|t|)^{-\al-n}\\
&\qquad+(-1)^m\vG{\al-\be}{\al}\sum_{n=0}^{N'-1}\fr{(\be)_n}{(1-\al+\be)_nn!}
\ph^{\ast}_{\be+n-m}(\si,a,\la)(e^{(\sgn t)\pi i/2}|t|)^{-\be-n}\\
&\qquad+O(|t|^{-\re\al-N})
\end{align*}
as $t\to\pm\infty$, where the implied $O$-constant depends at most on $\si$, $a$, $\la$, $\al$, $\be$, $m$ and $N$.
\end{corollary}
%%%%%%%%%%%%%%%%%%%%%%%%%%%%%%%%%%%%%%%%%%%%%%%%%%%%%%%%%%%%%%%%%%%%%%%%%%%%%%%%%%%%%%
%%%%%%%%%%

We next proceed to state the results on the iterations of the Laplace-Mellin and 
Riemann-Liouville transforms. For this, let $U(\ka;\nu;Z)$ be Kummer's confluent 
hypergeometric function of the second kind, defined by 
\begin{align*}
U(\ka;\nu;Z)
=\fr{1}{\vGa(\ka)}\int_0^{\infty}e^{-Z\xi}\xi^{\ka-1}(1+\xi)^{\nu-\ka-1}d\xi
\tag{3.8}
\end{align*}
in $|\arg Z|<\pi/2$ and for any complex $\ka$ and $\nu$ with $\re\ka>0$, where the 
domain of $Z$ can be extended to $|\arg Z|<3\pi/2$ by rotating appropriately the path 
of integration (cf.~\cite[p.237,~6.5(3)]{erdelyi1953a}). We can then show the following 
expression.
%%%%%%%%%%%%%%%%%%%%%%%%%%%%%%%%%%%%%%%%%%%%%%%%%%%%%%%%%%%%%%%%%%%%%%%%%%%%%%%%%%%%%%%%%%%%%%%%%%
\begin{proposition}
For any $f(z)$ holomorphic in $|\arg z|<\pi$, we have 
\begin{align*}
\RL_{z;\ta_2}^{\be,\ga}\LM_{\ta_2;\ta_1}^{\al}f(\ta_1)
=\vG{\be+\ga}{\al,\be}
\int_0^{\infty}f(z\ta)\ta^{\al-1}e^{-\ta}U(\ga;\al-\be+1;\ta)d\ta,
\tag{3.9}
\end{align*} 
provided that the integral converges, for which it suffices to choose a class of functions $f(z)$ 
such that $O(|z|^{-\max(\re\al,\re\be)+\vep})$ $(z\to0)$ and $O(e^{\re z}|z|^{\re(\ga-\al)-\vep})$ with any $\vep>0$ $(z\to\infty)$ both through $|\arg z|<\pi$.
\end{proposition}
%%%%%%%%%%%%%%%%%%%%%%%%%%%%%%%%%%%%%%%%%%%%%%%%%%%%%%%%%%%%%%%%%%%%%%%%%%%%%%%%%%%%%%
\begin{proof}
The left side of (3.9) equals
\begin{align*}
&\vG{\be+\ga}{\al,\be,\ga}
\int_0^{\infty}\int_0^{\infty}f(z\ta_1\ta_2)\ta_1^{\al-1}e^{-\ta_1}
\ta_2^{\be-1}(1-\ta_2)_+^{\ga-1}d\ta_1d\ta_2\\
&\quad=\vG{\be+\ga}{\al,\be,\ga}
\int_0^{\infty}f(z\ta)\ta^{\be-1}
\int_{\ta}^{\infty}e^{-\ta_1}\ta_1^{\al-\be-\ga}(\ta_1-\ta)_+^{\ga-1}d\ta_1d\ta,
\end{align*}
by setting $\ta_2=\ta/\ta_1$ on the first line; the resulting inner $\ta_1$-integral 
in the last expression is further evaluated as $\ta^{\al-\be}e^{-\ta}\vGa(\ga)
U(\ga;\al-\be+1;\ta)$, upon substituting the variable $\ta_1=\ta(1+\xi)$ 
and then by using (3.8), and this concludes (3.9). 
Here the (sufficient) condition for the convergence of the integral in (3.9) can be seen from the fact 
that the confluent hypergeometric fucntion in (3.9) is of order $\asymp\ta^{-\re\ga}$ $(\ta\to0^+$) and $\asymp \ta^{\max(0,\re(\be-\al))}$ $(\ta\to+\infty)$ 
(cf.~\cite[p.257, 6.5(7); p.278, 6.13.1(1)]{erdelyi1953a}). 
\end{proof}
%%%%%%%%%%%%%%%%%%%%%%%%%%%%%%%%%%%%%%%%%%%%%%%%

The following Theorems~7~and~8 assert the asymptotic expansions as $z\to0$ and 
$z\to\infty$ respectively. 
%%%%%%%%%%%%%%%%%%%%%%%%%%%%%%%%%%%%%%%%%%%%%%%%%%%%%%%%%%%%%%%%
\begin{theorem}%%%%%%%%%%%%%%%Theorem 7%%%%%%%%%%%%%%%%%%%%%%%%%%%%%
Upon the same settings as in Theorem~3 with any complex $\ga$ of positive real part, we have 
\begin{align*}
&\RL_{z;\ta_2}^{\be,\ga}\LM_{\ta_2;\ta_1}^{\al}
(\ph^{\ast})^{(m)}(s+\ta,a,\la)
\tag{3.10}\\
&\quad=(-1)^m\sum_{n=0}^{N-1}\fr{(-1)^{n}(\al)_n(\be)_n}{(\be+\ga)_nn!}
\ph^{\ast}_{-n-m}(s,a,\la)z^n
+R^{4,+}_{m,N}(s,a,\la;z)
\end{align*} 
in the sector $|\arg z|<\pi$. Here the reminder $R^{4,+}_{m,N}(s,a,\la;z)$ is expressed by the Mellin-Barnes type 
integral in (7.15) below, and satisfies the estimate 
\begin{align*}
R^{4,+}_{m,N}(s,a,\la;z)
=O\{(|t|+1)^{\max(0,\lf2-\si\rf)}|z|^N\}
\tag{3.11}
\end{align*}
as $z\to0$ through $|\arg z|\leq\pi-\et$, where the implied $O$-constant depends at most on $\si$, $a$, 
$\la$, $\al$, $\be$, $\ga$, $m$, $N$ and $\et$.
\end{theorem}%%%End of Theorem~7%%%%%%%%%%%%%%%%%%%%%%%%%%%%%
%%%%%%%%%%%%%%%%%%%%%%%%%%%%%%%%%%%%%%%%%%%%%%%%%%%%%%%%%%%%%%%%%%%%%%%%%%%%%%%%%%%%%%%%%%%%%%%
\begin{theorem}%%%%%%%%%%%%%%%%%%Theorem 8%%%%%%%%%%%%%%%%%%%%%%%%%
%%%%%%%%%%%%%%%%%%%%%%%%%%%%%%%%%%%%%%%%%%%%%%%%%%%%%%%%%%%%%%%%%%%
Upon the same settings as in Theorem~7, for any integers $N_j$ $(j=1,2)$ with 
$N_1\geq\lf\re(\be+\ga-\al)\rf$ and $N_2\geq\lf\re\ga\rf$, except the case $\al-\be\in\mathbb{Z}$ we have 
\begin{align*}
&\RL_{z;\ta_2}^{\be,\ga}\LM_{\ta_2;\ta_1}^{\al}
(\ph^{\ast})^{(m)}(s+\ta,a,\la)
\tag{3.12}\\
&\quad=
(-1)^m\vG{\be+\ga,\be-\al}{\be,\be+\ga-\al}
\Biggl\{\sum_{n=0}^{N_1-1}\fr{(-1)^{n}(\al)_n(1-\be-\ga+\al)_n}{(1-\be+\al)_nn!}\\
&\qquad\times\ph^{\ast}_{\al+n-m}(s,a,\la)z^{-\al-n}
+R_{1,m,N_1}^{4,-}(s,a,\la;z)\Biggr\}\\
&\qquad
+(-1)^m\vG{\be+\ga,\al-\be}{\al,\ga}
\Biggl\{\sum_{n=0}^{N_2-1}\fr{(-1)^{n}(\be)_n(1-\ga)_n}{(1-\al+\be)_nn!}\\
&\qquad\times\ph^{\ast}_{\be+n-m}(s,a,\la)z^{-\be-n}
+R^{4,-}_{2,m,N_2}(s,a,\la;z)\Biggr\}
\end{align*} 
in the sector $|\arg z|<\pi$. Here the reminders $R^{4,-}_{j,m,N_j}(s,a,\la;z)$ $(j=1,2)$ are expressed by the 
Mellin-Barnes type integrals in (7.22) below, and satisfy the estimates
\begin{align*}
\begin{split}
R^{4,-}_{1,m,N_1}(s,a,\la;z)
&=O\{(|t|+1)^{\max(0,\lf2-\si\rf)}|z|^{-\re\al-N_1}\},
\\
R^{4,-}_{2,m,N_2}(s,a,\la;z)
&=O\{(|t|+1)^{\max(0,\lf2-\si\rf)}|z|^{-\re\be-N_2}\}
\end{split}\tag{3.13}
\end{align*}
as $z\to\infty$ through $|\arg z|\leq\pi-\et$ with any small $\et>0$, where the implied 
$O$-constants depend at most on $\si$, $a$, $\la$, $\al$, $\be$, $\ga$, $m$, $N$ and $\et$. 
\end{theorem}%%%%%End of Thorem~8%%%%%%%%%%%%%
%%%%%%%%%%%%%%%%%%%%%%%%%%%%%%%%%%%%%%%%%%%%%%%%%%%%%%%%%%%%%%%%%%%%%%%%%%%%%%%%%%%%%%%%%%%%%
%%%%%%%%%%%%%%%%%%%%%%%%%%%%%%%%%%%%%%%%%%%%%%%%%%%%%%%%%%%%%%%%%%%%%%%%%%%%%%%%%%%%%%%%%%%%%
The case $(s,z)=(\si,it)$ with $\si,t\in\mathbb{R}$ of Theorem~8 asserts the following 
asymptotic expansion along vertical lines.  
%%%%%%%%%%%%%%%%%%%%%%%%%%%%%%%%%%%%%%%%%%%%%%%%%%%%%%%%%%%%%%%%%%%%%%%%%%%%%%
%%%%%%Corollary~8.1%%%%%%%%%%%%%%%%%%%%%%%%%%%%%%%%%%%%%%%%%%%%%%%%%%%%%%%%%%%%%%%
\begin{corollary}%%Corollary~8.1%%%%%%%%%%%%%%%
Upon the same settings as in Theorem~8 wtih any real $\si$, we have 
\begin{align*}
&\RL_{t;\ta_2}^{\be,\ga}\LM_{\ta_2;\ta_1}^{\al}
(\ph^{\ast})^{(m)}(\si+i\ta,a,\la)
\tag{3.14}\\
&\quad=
(-1)^m\vG{\be+\ga,\be-\al}{\be,\be+\ga-\al}
\Biggl\{\sum_{n=0}^{N_1-1}\fr{(-1)^{n}(\al)_n(1-\be-\ga+\al)_n}{(1-\be+\al)_nn!}\\
&\qquad\times\ph^{\ast}_{\al+n-m}(s,a,\la)(e^{(\sgn t)\pi i/2}|t|)^{-\al-n}
+O(|t|^{-\re\al-N_1})\Biggr\}\\
&\qquad
+(-1)^m\vG{\be+\ga,\al-\be}{\al,\ga}
\Biggl\{\sum_{n=0}^{N_2-1}\fr{(-1)^{n}(\be)_n(1-\ga)_n}{(1-\al+\be)_nn!}\\
&\qquad\times\ph^{\ast}_{\be+n-m}(s,a,\la)(e^{(\sgn t)\pi i/2}|t|)^{-\be-n}
+O(|t|^{-\re\be-N_2})\Biggr\}
\end{align*} 
as $t\to\pm\infty$, where the implied $O$-constants depend at most on $\si$, $a$, $\la$, $\al$, $\be$, $\ga$, $m$, $N$ and $\et$. 
\end{corollary}

%%%%%%%%%%%%%%%%%%%%%%%%%%%%%%%%%%%%%%%%%%%%%%%%%%%%%%%%%%%%%%%%%%%%%%%%%%%%%%%%%%%%%%
%%%%%%

We next proceed to state the results on the iteration of two Riemann-Liouville 
transforms. For this, let $\tF{\ka,\mu}{\nu}{Z}$ be Gau{\ss}' hypergeometric 
function defined by    
\begin{align*}
\F{\ka,\mu}{\nu}{Z}
=\sum_{n=0}^{\infty}\fr{(\ka)_n(\mu)_n}{(\nu)_nn!}Z^n
\qquad(|Z|<1)\tag{3.15}
\end{align*}
for any complex $\ka$, $\mu$ and $\nu$ with  $\nu\notin\mathbb{Z}_{\leq0}$ 
(cf.~\cite[p.56,~2.1.1(2)]{erdelyi1953a}), where the domain of $Z$ is extended to 
$|\arg(1-Z)|<\pi$ by Euler's integral formula (3.17) below. 
We can then show the following expression. 
\begin{proposition}
For any $f(\ta)$ holomorphic in $|\arg\ta|<\pi$, we have
\begin{align*}
\RL_{z;\ta_2}^{\ga,\de}\RL_{\ta_2;\ta_1}^{\al,\be}f(\ta_1)
&=\vG{\al+\be,\ga+\de}{\al,\ga,\be+\de}
\int_0^{\infty}f(z\ta)\ta^{\al-1}(1-\ta)_+^{\be+\de-1}
\tag{3.16}\\
&\quad\times\F{\al+\be-\ga,\de}{\be+\de}{1-\ta}d\ta,
\end{align*} 
provided that the integral converges, for which it suffuces to choose a class of 
functions such that $f(z)=O(1)$ as $z\to0$ through $|\arg z|<\pi$.
\begin{proof}
The left side of (3.16) equals
\begin{align*}
&\vG{\al+\be,\ga+\de}{\al,\be,\ga,\de}
\int_0^{\infty}\int_0^{\infty}f(z\ta_1\ta_2)\ta_2^{\ga-1}(1-\ta_2)_+^{\de-1}
\ta_1^{\al-1}(1-\ta_1)_+^{\be-1}d\ta_1d\ta_2\\
&\quad
=\vG{\al+\be,\ga+\de}{\al,\be,\ga,\de}
\int_0^{\infty}f(z\ta)\ta^{\al-1}
\int_0^{\infty}\ta_2^{\ga-\al-\be}(1-\ta_2)_+^{\de-1}(\ta_2-\ta)_+^{\be-1}d\ta_2d\ta, 
\end{align*}
by setting $\ta_1=\ta/\ta_2$ and then by interchanging the order of the integrals 
on the first line; the resulting inner $\ta_2$-integral further becomes
\begin{align*}
&\int_{\ta}^1\ta_2^{\ga-\al-\be}(1-\ta_2)_+^{\de-1}(\ta_2-\ta)_+^{\be-1}d\ta_2\\
&\quad=\int_0^1\{1-(1-\ta)\xi\}^{\ga-\al-\be}\{(1-\ta)\xi\}_+^{\de-1}
\{(1-\ta)(1-\xi)\}_+^{\be-1}(1-\ta)d\xi\\
&\quad=(1-\ta)_+^{\be+\de-1}\vG{\be,\de}{\be+\de}
\F{\al+\be-\ga,\de}{\be+\de}{1-\ta},   
\end{align*} 
where the substitution of the variable $\ta_2=1-(1-\ta)\xi$ is made on the first line, while 
the integral on the penultimate line is evaluated by the formula
\begin{align*}
\F{\ka,\mu}{\nu}{Z}
=\vG{\nu}{\mu,\nu-\mu}\int_0^1\xi^{\mu-1}(1-\xi)^{\nu-\mu-1}(1-Z\xi)^{-\ka}d\xi
\tag{3.17}
\end{align*}
in $|\arg(1-Z)|<\pi$ and for any complex $\ka$, $\mu$ and $\nu$ with $\re\nu>\re\mu>0$ (cf.~\cite[p.59, 2.1.3(10)]{erdelyi1953a}); this concludes (3.16). Here the sufficient condition for the 
convergence of the integral in (3.16) can be seen from the fact 
that the hypergeometric function in (3.16) is of order $\asymp 1+\ta^{\re(\ga-\al)}$ $(\ta\to0^+)$ and 
$\asymp 1$ $(\ta\to1^{-})$ (cf.~(3.15) and \cite[p.108, 2.10(1)]{erdelyi1953a}).
\end{proof}
%%%%%%%%%%%%%%%%%%%%%%%%%%%%%%%%%%%%%%%%%%%%%%%%%%%%%%%%%%%%%%%%%%%%%%%%%%%%%%
\end{proposition}
%%%%%%%%%%%%%%%%%%%%%%%%%%%%%%%%%%%%%%%%%%%%%%%%%%%%%%%%%%%%%%%%%%%%%%%%%%%%%%%%%%%%
The following Theorems~9~and~10 assert the asymptotic expansions as $z\to0$ and 
$z\to\infty$ respectively. 
%%%%%%%Theorem~9%%%%%%%%%%%%%%%%%%%%%%%%%%%%%%%%%%%%%%
\begin{theorem}%%%%%%%%%%%%%%%%Theorem 9%%%%%%%%%%%%%%%%%%%%%%%%%%%%
Upon the same settings as in Theorem~7 with any complex $\de$ of positive real part, we have 
\begin{align*}
&\RL_{z;\ta_2}^{\ga,\de}\RL_{\ta_2;\ta_1}^{\al,\be}
(\ph^{\ast})^{(m)}(s+\ta_1,a,\la)
\tag{3.18}\\
&\quad=(-1)^m\sum_{n=0}^{N-1}
\fr{(-1)^{n}(\al)_n(\ga)_n}{(\al+\be)_n(\ga+\de)_nn!}
\ph^{\ast}_{-n-m}(s,a,\la)z^n
+R^{5,+}_{m,N}(s,a,\la;z)
\end{align*}
in the sector $|\arg z|<\pi/2$. Here the reminder $R^{5,+}_{m,N}(s,a,\la;z)$ is expressed by the Mellin-Barnes type 
inetgral in (7.31) below, and satisfies the estimate
\begin{align*}
R^{5,+}_{m,N}(s,a,\la;z)
=O\{(|t|+1)^{\max(0,\lf2-\si\rf)}|z|^N\}
\tag{3.19}
\end{align*} 
as $z\to0$ through $|\arg z|\leq\pi/2-\et$ with any small $\et>0$, where the implied $O$-constant 
depends at most on $\si$, $a$, $\la$, $\al$, $\be$, $\ga$, $\de$, $m$, $N$ and $\et$.  
\end{theorem}
%%%%%%%%%%%%%%%%%%%%%%%%%%%%%%%%%%%%%%%%%%%%%%%%%%%%%%%%%%%%%%%%%%%%%%%%%%%%%%%%%%%%%%
\begin{theorem}%%%%%%%%%%%%%Therorem 10%%%%%%%%%%%%%%%%%%%%%%%%%%%%%%%%%%%%%%
Upon the same settings as in Theorem~9 with $N'=N-\lf\re(\be-\al)\rf$, except the case 
$\al-\ga\in\mathbb{Z}$ we have   
\begin{align*}
&\RL_{s;\ta_2}^{\ga,\de}\RL_{\ta_2;\ta_1}^{\al,\be}
(\ph^{\ast})^{(m)}(s+\ta_1,a,\la)
\tag{3.20}\\
&\quad=
(-1)^m\vG{\al+\be,\ga-\al,\ga+\de}{\be,\ga,\ga+\de-\al}\\
&\qquad\times\sum_{n=0}^{N-1}
\fr{(\al)_n(1-\be)_n(1+\al-\ga-\de)_n}{(1+\al-\ga)_nn!}
\ph^{\ast}_{\al+n-m}(s,a,\la)z^{-\al-n}\\
&\qquad+(-1)^m\vG{\al+\be,\al-\ga,\ga+\de}{\al,\al+\be-\ga,\ga+\de-\be}\\
&\qquad\times\sum_{n=0}^{N'-1}
\fr{(\be)_n(1-\al-\be+\ga)_n(1+\be-\ga-\de)_n}{(1-\al+\ga)_nn!}
\ph^{\ast}_{\be+n-m}(s,a,\la)z^{-\be-n}\\
&\qquad+R^{5,-}_{m,N}(s,a,\la;z)
\end{align*}
in the sector $|\arg z|<\pi/2$. Here the reminder $R^{5,-}_{m,N}(s,a,\la;z)$ is expressed by the 
Mellin-Barnes type integral in (7.32) below, and satisfies the estimate
\begin{align*}
R^{5,-}_{m,N}(s,a,\la;z)
=O\{(|t|+1)^{\max(0,\lf2-\si\rf)}|z|^{-\re \al-N}\}
\tag{3.21}
\end{align*} 
as $z\to\infty$ through $|\arg z|\leq\pi/2-\et$ with any small $\et>0$, where the implied $O$-constant 
depends at most on $\si$, $a$, $\la$, $\al$, $\be$, $\ga$, $\de$, $m$, $N$ and $\et$.    
%%%%%%%%%%%%%%%%%%%%%%%%%%%%%%%%%%%%%%%%%%%%%%%%%%%%
\end{theorem}
%%%%%%%%%%%%%%%%%%%%%%%%%%%%%%%%%%%%%%%%%%%%%%%%%%%%%%%%%%%%%%%%%%%%%%%%%%%%%%%%%%%%%%%%%%%%%%%%%%
\begin{remark}
Theorem~10 fails to imply the complete asymptotic expansion for 
\begin{align*}
\RL_{t;\ta_2}^{\ga,\de}\RL_{\ta_2;\ta_1}^{\al,\be}(\ph^{\ast})^{(m)}(\si+i\ta_1,a,\la)
\end{align*}
as $t\to\pm\infty$, since any vertical lines are not included in the region of its validity.
\end{remark}  
\begin{problem}
Deduce the complete asymptotic expansion for  
\begin{align*}
\RL_{s;\ta_2}^{\ga,\de}\RL_{\ta_2;\ta_1}^{\al,\be}(\ph^{\ast})^{(m)}(s+\ta_1,a,\la),
\end{align*}
which is valid in the full sector $|\arg z|<\pi$. 
\end{problem} 
It seems to be suggestive for the solution to introduce further the auxiliary zeta-function 
$\RL_{t;\ta}^{\be,\ga}\ph^{\ast}_{r}(s+\ta,a,\la)$, and to investigate its properties. 
%%%%%%%%%%%%%%%%%%%%%%%%%%%%%%%%%%%%%%%%%%%%%%%%%%%%%%%%%%%%%%%%%%%%%%%%%%%%%%%%%%%
\section{Preliminaries}%%Section 4
%%%%%%%%%%%%%%%%%%%%%%%%%%%%%%%%%%%%%%%%%%%%%%%%%%%%%%%%%%%%%%%%%%%%%%%%%%%%%%%%%%%
Throughout this section, we write $r=\rh+i\ta$ and $s=\si+it$ with real coordinates $\rh$, $\si$, $\ta$  
and $t$. We first show in this section the formulae (4.3) and (4.4), which apropriates to deduce the 
key estimate (4.12). 

Let $\fs_m^n(x)$ denote Stirling's polynomial of the first kind, defined for any integers $m,n\geq0$ by 
\begin{align*}
\fs_m^n(x)=\left.\fr{1}{m!}\Bigl(\fr{\del}{\del z}\Bigr)^n(1-z)^{-x}\{-\log(1-z)\}^m\right|_{z=0},
\tag{4.1}
\end{align*}
from which the fact $\fs_m^n(x)=0$ follows for any $m>n\geq0$. 
\begin{proposition}%%%%%%%%%%%%Proposition 4%%%%%%%%%%%%%%%%%%%%%%%%%%%%%%
For any $m,n\geq0$ we have 
\begin{align*}
(x+y)_n=\sum_{m=0}^{n}\fs_m^n(x)y^m.
\tag{4.2}
\end{align*} 
\end{proposition}
\begin{proof}%%%%%%%%%%%%%%Proof of Proposition 4%%%%%%%%%%%%%%%%%%%%%%%%%%%%
The left side of (4.2) equals, by the Taylor series expansion (in terms of $y$),  
\begin{align*}
\Bigl(\fr{\del}{\del z}\Bigr)^n(1-z)^{-x-y}\biggr|_{z=0}
&=\sum_{m=0}^{\infty}
\Bigl(\fr{\del}{\del z}\Bigr)^n(1-z)^{-x}\{-\log(1-z)\}^m\biggr|_{z=0}\fr{y^m}{m!},
\end{align*}
the right side of which gives that of (4.2) with the expresion in (4.1).  
\end{proof}
The relation (4.2) readily shows that $\deg_x\fs_m^n(x)=n-m$ for any $n\geq m\geq0$. 
\begin{lemma}%%%%%%%%%%%%%%%%%%%%Lemma 1%%%%%%%%%%%%%%%%%%%%%%%%%%5
For any complex $r$ and $s$, for any real $a$ and $\la$ with $a>1$, and for any integer $N\geq1$, 
we have{\rm:}
\begin{itemize}
\item[i)] if $\la\notin\mathbb{Z}$,
\begin{align*}
&\ph^{\ast}_r(s,a,\la)\tag{4.3}\\
&\quad=\fr{a^{-s}\log^{-r} a}{1-e(\la)}+\fr{e(\la)}{1-e(\la)}
\Biggl\{\sum_{n=1}^{N-1}\fr{(-1)^n}{n!}
\sum_{m=0}^{n}\fs_m^n(s)(r)_m\ph^{\ast}_{r+m}(s+n,a,\la)\\
&\qquad+\fr{(-1)^{N-1}}{(N-1)!}\sum_{m=0}^{N}\fs_m^N(s)(r)_m
\int_0^1\ph^{\ast}_{r+m}(s+N,a+u,\la)(1-u)^Ndu\Biggr\};
\end{align*} 
\item[ii)] if $\la\in\mathbb{Z}$, 
\begin{align*}
\ze_{r}^{\ast}(s,a)
&=\int_0^1(a+u)^{-s-1}\log^{-r}(a+u)\cdot(1-u)du\tag{4.4}\\
&\quad+\sum_{n=1}^{N-1}
\fr{(-1)^{n+1}}{(n+1)!}\sum_{m=0}^{n}\fs_m^n(s)(r)_m
\ze^{\ast}_{r+m}(s+n,a)\\
&\quad+\fr{(-1)^{N+1}}{N!}\sum_{m=0}^{N-1}\fs_m^N(s)(r)_m
\int_0^1\ze^{\ast}_{r+m}(s+N,a+u)(1-u)^Ndu.
\end{align*}
\end{itemize}  
%%%%%%%%%%%%%%%%%%%%%%%%%%%%%%%%%%%%%%%%%%%%%%%%%%%%%%%%%%%%%%%%%%%%%%%%%%%%%%%%%%%%%%%%%%%%%%%
\end{lemma}
%%%%%%%%%%%%%%%%%%%%%%%%%%%%%%%%%%%%%%%%%%%%%%%%%%%%%%%%%%%%%%%%%%%%%%%%%%%%%%%%%%%%%%%%%%%%%%%%%%%%
\begin{proof}%%%%%%%%%%%%%%%%%%%%%%%%%%Proof of Lemma 1%%%%%%%%%%%%%%%%%%%%%%%
The defining series in (1.1) readily implies the relations 
\begin{align*}
\ph(s,a,\la)=a^{-s}+e(\la)\ph(s,a+1,\la),\tag{4.5}
\end{align*}
%%%%
\begin{align*}
\Bigl(\fr{\del}{\del z}\Bigr)^n\ph(s,z,\la)
=(-1)^n(s)_n\ph(s+n,z,\la)
\qquad(n=0,1,\ldots),\tag{4.6}
\end{align*}
where the latter holds for any complex $s\neq1$ and in $|\arg z|<\pi$ by analytic continuation; 
this is combined with the relation, by (1.4), 
\begin{align*}
\Bigl(\fr{\del}{\del z}\Bigr)^n\ps(s,z)=(-1)^n(s)_n\ps(s+n,z)
\qquad(n=0,1,\ldots)
\tag{4.7}
\end{align*} 
to imply upon (1.5) that 
\begin{align*}
\Bigl(\fr{\del}{\del z}\Bigr)^n\ph^{\ast}(s,z,\la)
=(-1)^n(s)_n\ph^{\ast}(s+n,z,\la)
\qquad(n=0,1,\ldots).\tag{4.8}
\end{align*}

We now treat Case i), where $\ph^{\ast}(s,a,\la)=\ph(s,a,\la)$ holds by (1.5). Applying 
Taylor's formula (centered at $z=a$) to expand $\ph^{\ast}(s,a+1,\la)$ on the right side of 
(4.5), and then subtructing $e(\la)\ph^{\ast}(s,a,\la)$ from both sides, we see for any 
integer $N\geq1$ that 
\begin{align*}
\{1-e(\la)\}\ph^{\ast}(s,a,\la)
&=a^{-s}
+e(\la)\Biggl\{
\sum_{n=1}^{N-1}\fr{(-1)^n(s)_n}{n!}\ph^{\ast}(s+n,a,\la)\\
&\quad+\fr{(-1)^N(s)_N}{(N-1)!}\int_0^1\ph^{\ast}(s+N,a+u,\la)(1-u)^{N-1}du\Biggr\}.
\end{align*} 
We further operate $\I{r}{s}$ (see (2.1)) to both sides above, and then note that 
\begin{align*}
\I{r}{s}a^{-s}=a^{-s}\log^{-r} a\tag{4.9}
\end{align*}
for any $(r,s)\in\mathbb{C}^2$ and for $a>1$, and further that 
the resulting form of $\I{r}{s}\{(s)_n\ph^{\ast}(s,a,\la)\}$ in the $n$th indexed term 
equals, by (4.2),  
\begin{align*}
&\sum_{m=0}^{n}\fs_m^n(s)\cdot\fr{1}{\vGa(r)\{e(r)-1\}}
\int_{\infty}^{(0+)}\ph^{\ast}(s+n+z,a,\la)z^{r+m-1}dz\tag{4.10}\\
&\quad=\sum_{m=0}^n\fs_m^n(s)
\fr{\vGa(r+m)\{e(r+m)-1\}}{\vGa(r)\{e(r)-1\}}
\ph^{\ast}_{r+m}(s+n,a,\la),
\end{align*}
to conclude (4.3).

We next treat Case ii). Suppose temporarily that $\si>1$. It follows from (4.5) and the expression
\begin{align*}
\ps(s,z)
&=\int_z^{z+\infty}\xi^{-s}d\xi
=\int_0^{+\infty}(z+u)^{-s}du,\tag{4.11}
\end{align*}
where the path of integration is the horizontal ray, that 
\begin{align*}
\ze^{\ast}(s,a+1)&=\ze^{\ast}(s,a)-a^{-s}+\int_a^{a+1}\xi^{-s}d\xi\\
&=\ze^{\ast}(s,a)-s\int_0^1(a+u)^{-s-1}(1-u)du.
\end{align*}
We now apply Taylor's fomula (centered at $z=a$) to expand $\ze^{\ast}(s,a+1)$ on the left side 
above, then subtruct the term $\ze^{\ast}(s,a)$ from both sides, cancelling out the factor 
$s$ (upon noting $(s)_n=s(s+1)_{n-1}$ for $n\geq1$), and further replace $(s,n,N)$ by 
$(s-1,n+1,N+1)$ in the resulting expression, to find for any $N\geq1$ that
\begin{align*}
\ze^{\ast}(s,a)
&=\int_0^1(a+u)^{-s-1}(1-u)du
+\sum_{n=1}^{N-1}\fr{(-1)^{n+1}(s)_n}{(n+1)!}\ze^{\ast}(s+n,a)\\
&\quad+\fr{(-1)^{N+1}(s)_N}{N!}\int_0^1\ze^{\ast}(s+N,a+u)(1-u)^Ndu,
\end{align*}
which is valid for all $s\in\mathbb{C}$ by analytic continuation. Operating $\I{r}{s}$ to both sides above, and noting (4.2), (4.9) and (4.10) again, we conclude (4.4).    
%%%%%%%%%%%%%%%%%%%%%%%%%%%%%%%%%%%%%%%%%%%%%%%%%%%%%%%%%%%%%%%%%%%%%%%%%%%%%%%%%%%%%%%%%%%%%%%%%
\end{proof}
%%%%%%%%%%%%%%%%%%%%%%%%%%%%%%%%%%%%%%%%%%%%%%%%%%%%%%%%%%%%%%%%%%%%%%%%%%%%%%%%%%%%%%%%%%%%%%%%%%%%%%Lemma~2%%%%%%%%%%%
\begin{lemma}%%%%%%%%%%%%%%%%%%%%%%%%Lemma 2%%%%%%%%%%%%%%%%%%%%%%%%%%%%%%%%%%%%%%%
For any complex $r$ and $s$, and any real $a$ and $\la$ with $a>1$, the vertical estimate
\begin{align*}
\ph^{\ast}_r(s,a,\la)
\ll(|\ta|+|t|+1)^{\max(0,\lf2-\si\rf)},\tag{4.12}
\end{align*}
holds, where the implied $\ll$-constant depends at most on $\rh$, $\si$, $a$ and $\la$.
\end{lemma}
\begin{proof}%%%%%%%%%%%%%%%%%%%%%%%%%%%Proof of Lemma 2%%%%%%%%%%%%%%%%%%%%%%%%%%%
Suppose temporarily that $\si>1$, and write
\begin{align*}
\ph_r(s,a,\la)=\I{r}{s}\ph(s,a,\la)
\qquad
\text{and}
\qquad
\ps_r(s,a)=\I{r}{s}\ps(s,a)
\tag{4.13}
\end{align*}
%5
for $a>1$. We then observe that the term-by-term application of $\I{r}{s}$ on the right sides of (1.1) and (4.11) shows for $\si>1$ that, by (4.9),  
\begin{align*}
\begin{split}
\ph_r(s,a,\la)
&=\sum_{l=0}^{\infty}e(\la l)(a+l)^{-s}\log^{-r}(a+l),\\
\ps_r(s,a)&=\int_a^{+\infty}\xi^{-s}\log^{-r}\xi d\xi.
\end{split}
\tag{4.14}
\end{align*}   

Consider first the case of $\si>1$. It is then seen from (1.5) and (4.14) that 
\begin{align*}
\ph^{\ast}_r(s,a,\la)
&=\ph_r(s,a,\la)+\de_{\mathbb{Z}}(\la)\ps_r(s,a)\ll 1,
\tag{4.15}
\end{align*}
which just gives (4.12) for $\si>1$.

Consider next the case of $1-N<\si\leq 2-N$ with any integer $N\geq1$. We shall then 
prove (4.12) by induction on $N$; the induction hypothesis in this case is that the inequality
\begin{align*}
\ph^{\ast}_r(s,a,\la)
\ll(|\ta|+|t|+1)^n
\tag{4.16}
\end{align*}
holds for $1-n<\si\leq 2-n$ with $n=1,2,\ldots,N-1$. Noting that $1-(N-n)<\si+n\leq2-(N-n)$ in the present case, we substitute the bound in (4.17) with $(s,n)$ replaced by $(s+n,N-n)$, and also the bounds  $\fs_m^n(s)\ll (|t|+1)^{n-m}$ and $(r)_m\ll(|\ta|+1)^m$ into the right side of (4.3) or (4.4), 
to find that 
\begin{align*}
\ph^{\ast}_r(s,a,\la)
&\ll 1+\sum_{n=1}^{N-1}\sum_{m=0}^n(|t|+1)^{n-m}(|\ta|+1)^m(|\ta|+|t|+1)^{N-n}\\
&\ll(|\ta|+|t|+1)^N,
\end{align*}    
where the condition $1-N<\si\leq 2-N$ is equivalent to $N=\lf2-\si\rf$; this with (4.15) concludes (4.12). 
\end{proof}
%%%%%%%%%%%%%%%%%%%%%%%%%%%%%%%%%%%%%%%%%%%%%%%%%%%%%%%%%%%%%%%%
\begin{lemma}%%%%%%%%%%%%%%%%%%%%%%%%%%%Lemma 3%%%%%%%%%%%%%%%%%%%%%%%%%%%%%%%%%%
%%%%%%%%%%%%%%%%%%%%%%%%%%%%%%%%%%%%%%%%%%%%%%%%%%%%%%%%%%%%%%%
Let $f(s)$ is a function horomorphic at any $s\in\mathbb{C}\setminus]-\infty,1]$. For any 
$m\in\bbZ$, if $f(s+x)$ belongs {\rm(}as a function of $x${\rm)} to the class 
$x^{\max(1+m,0)}L_x^1[0,+\infty[$, then 
\begin{align*}
f^{(m)}(s)=(-1)^m\I{-m}{s}f(s).
\tag{4.17}
\end{align*} 
\end{lemma}
%%%%%%%%%%%%%%%%%%%%%%%%%%%%%%%%%%%%%%%%%%%%%%%%%%%%%%%%%%%%%%%%%%%%%%%%%%%%%%
\begin{proof}%%%%%%%%%%%%%%%%%%%%Proo of Lemma 4%%%%%%%%%%%%%%%%%%%%%%%%%%%%%%%%%
Suppose first that $m\geq0$. Then the limiting case $r\to-m$ of (2.1)  shows  
\begin{align*}
\I{-m}{s}f(s)=\fr{(-1)^mm!}{2\pi i}\oint_{|z|=\et}\fr{f(s+z)}{z^{m+1}}dz=(-1)^mf^{(m)}(s),
\end{align*}
where $\et>0$ is taken so as to be smaller than the distance between $s$ and the half-line $]-\infty,1]$. 

Next the assertion for $m=-n<0$ $(n=1,2,\ldots)$ is proved by the induction on $n$. We see for $n=1$ that 
\begin{align*}
\I{1}{s}f(s)=-\int_{0}^{\infty}f(s+x)dx=-f^{(-1)}(s),
 \end{align*}
since $\oint_{|z|=\de}f(s+z)dz\to0$ as $\de\to0^+$. Suppose now that 
\begin{align*}
f^{(-n+1)}(s)=\fr{(-1)^{n-1}}{\vGa(n-1)}\int_s^{\infty}f(w)(w-s)^{n-2}dw
=(-1)^{n-1}\I{n-1}{s}f(s)
\end{align*} 
for $n\geq2$. Then we have 
\begin{align*}
f^{(-n)}(s)&=-\int_s^{\infty}\fr{(-1)^{n-1}}{\vGa(n-1)}
\int_{w'}^{\infty}f(w)(w-w')^{n-2}dwdw'\\
&=\fr{(-1)^n}{\vGa(n)}\int_s^{\infty}f(w)(w-s)^{n-1}dw
=(-1)^n\I{n}{s}f(s), 
\end{align*}
which is the assertion for $m=-n$. Lemma 3 is thus proved.
\end{proof}
%%%%%%%%%%%%%%%%%%%%%%%%%%%%%%%%%%%%%%%%%%%%%%%%%%%%%%%%%%%%%%%%%%%%%%%%%%%%%%
%%%%%%%%%%%%%%%%%%%%%%%%%%%%%%%%%%%%%%%%%%%%%%%%%%%%%%%%%%%%%%%%%%%%%%%%%%%%%%
\section{Asymptotic expansions for the Laplace-Mellin transform}%%%%%%Section 5%%%%%%%%%%%
%%%%%%%%%%%%%%%%%%%%%%%%%%%%%%%%%%%%%%%%%%%%%%%%%%%%%%%%%%%%%%%%%%%%%%%%%%%%%%
%%%%%%%%%%%%%%%%%%%%%%%%%%%%%%%%%%%%%%%%%%%%%%%%%%%%%%%%%%%%%%%%%%%%%%%%%%%%%%
We shall prove Theorems~1~and~2 in this section.

Suppose temporarily that $\si>1$ and $|\arg z|<\pi/2$. We then operate $\LM_{z;\ta}^{\al}$ 
(see (1.2)) term-by-term on the right sides of (4.13), upon noting (4.17), that   
\begin{align*}
\begin{split}
\LM_{z;\ta}^{\al}\ph^{(m)}(s+\ta,a,\la)
&=(-1)^m\sum_{l=0}^{\infty}e(\la l)(a+l)^{-s}\log^m(a+l)\\
&\quad\times\{1+z\log(a+l)\}^{-\al},\\
\LM_{z;\ta}^{\al}\ps^{(m)}(s,a)
&=(-1)^m\int_a^{\infty}\xi^{-s}\log^m\xi\cdot(1+z\log\xi)^{-\al}d\xi
\end{split}
\tag{5.1}
\end{align*}
for any $m\in\mathbb{Z}$, where the resulting expressions converge absolutely. Let $(u)$ for $u\in\mathbb{R}$ denote the vertical path from $u-i\infty$ to $u+i\infty$, and write $w=u+iv$ with real 
coordinates $u$ and $v$ throughout the sequel. The Mellin-Barnes 
formula 
\begin{align*}
(1+Z)^{-\al}
=\fr{1}{2\pi i}\int_{(u)}\vG{\al+w,-w}{\al}Z^wdw
\tag{5.2}
\end{align*} 
holds for $|\arg Z|<\pi$ with $-\re\al<u<0$; this is incorporated in each term on the right sides of (5.1), and then the order of the $w$-integral and the $l$-sum or $\xi$-integral is interchanged to 
show that
\begin{align*}
\begin{split}
\LM_{z;\ta}^{\al}\ph^{(m)}(s+\ta,a,\la)
&=\fr{(-1)^m}{2\pi i}\int_{(u_0)}\vG{\al+w,-w}{\al}\ph_{-w-m}(s,a,\la)z^wdw,\\
\LM_{z;\ta}^{\al}\ps^{(m)}(s,a)
&=\fr{(-1)^m}{2\pi i}\int_{(u_0)}\vG{\al+w,-w}{\al}\ps_{-w-m}(s,a)z^wdw
\end{split}
\tag{5.3}
\end{align*}
(see (4.13)) for any $m\in\mathbb{Z}$ with $-\re\al<u_0<0$. 
We therefore obtain upon (1.5) the following lemma.
%%Lemma~4%%%%%%%%%%%%%%%%
\begin{lemma}%%Lemma~4%%%%%%%%%%%%%%%%%%%%
The formula
\begin{align*}
\LM_{z;\ta}^{\al}
(\ph^{\ast})^{(m)}(s+\ta,a,\la)
&=\fr{(-1)^m}{2\pi i}\int_{(u_0)}\vG{\al+w,-w}{\al}\ph^{\ast}_{-w-m}(s,a,\la)z^wdw
\tag{5.4}
\end{align*}
holds for any $m\in\mathbb{Z}$ with a constant $u_0$ satisfying $-\re\al<u_0<0$, where the integral on the right side converges absolutely for all $(s,z)$ with $s\in\mathbb{C}$ and $|\arg z|<\pi${\rm;} this provides the analytic continuation of the left side to the same region of $(s,z)$. 
\end{lemma}
\begin{proof}
Stirling's formula for the gamma function implies that 
\begin{align*}
\vGa(s)\asymp (|t|+1)^{\si-1/2}e^{-\pi|t|/2}
\tag{5.5}
\end{align*}
for any $s\in\mathbb{C}$ apart from the ploes at $s=-n$ $(n=0,1,\ldots)$ (cf.~\cite[p.492, A.7(A.34)]{ivic1985}). Then the integrand in (5.4) on the line $\re w=u$ is (apart from its poles) bounded  by (4.12) and (5.5) as 
\begin{align*}
\ll e^{-|v|(\pi-|\arg z|)}(|v|+1)^{\re\al-1}(|v|+|t|+1)^{\max(0,\lf2-\si\rf)}
\tag{5.6}
\end{align*}
which asserts the required absolute convergence.
\end{proof}
%%%%%%%%%%%%%%%%%%%%%%%%%%%%%%%%%%%%%%%%%%%%%%%%%%%%%%%%%%%%%%%%%%%%%%%%%%%%%%%%%%%%%
We are now ready to prove Theorems~1~and~2. 
%%%%%%%%%%%%%%%%%%%%%%%%%%%%%%%%%%%%%%%%%%%%%%%%%%%%%%%%%%%%%%%%%%%%%%%5
\begin{proof}[Proof of Theorem~1]
%%%%%%%%%%%%%%%%%%%%%%%%%%%%%%%%%%%%%%%%%%%%%%%%%%%%%%%%%%%%%%%%%%%%%%%%
Let $N\geq0$ be any integer, and choose a constant $u_N^+$ such as  
$\max(-\re\al,N-1)<u_N^+<N$. We can then move upon (5.6) the path of integration in (5.4) 
to the right from $(u_0)$ to $(u_N^+)$, passing over the poles of the integrand at $s=n$ 
$(n=0,1,\ldots, N-1)$. Collecting the residues of the relevant poles, we find that the 
expression (2.3) holds with 
\begin{align*}
R_{m,N}^{1,+}(s,a,\la;z)
=\fr{(-1)^m}{2\pi i}\int_{(u_N^+)}\vG{\al+w,-w}{\al}
\ph^{\ast}_{-w-m}(s,a,\la)z^wdw.
\tag{5.7}
\end{align*}

The remaining estimate (2.4) is obtained by moving further the path in (5.7) from $(u_N^+)$ 
to $(u_{N+1}^+)$; this gives
\begin{align}
R_{m,N}^{1,+}(s,a,\la;z)
=\fr{(-1)^N(\al)_N}{N!}\ph^{\ast}_{-N-m}(s,a,\la)z^N
+R_{m,N+1}^{1,+}(s,a,\la;z).
\tag{5.8}
\end{align}
Here on the right side, the first term is estimated as $\ll(|t|+1)^{\max(0,\lf2-\si\rf)}|z|^N$ by (4.12), 
while the second, upon using (5.6) and (5.7) with $u_{N+1}^+$ instead of $u_N^+$, as
\begin{align*}
&\ll\int_{-\infty}^{\infty}e^{-|v|(\pi-|\arg z|)}(|v|+1)^{\re\al-1}
(|t|+|v|+1)^{\max(0,\lf2-\si\rf)}|z|^{u_{N+1}^+}dv\\
&\ll(|t|+1)^{\max(0,\lf2-\si\rf)}|z|^{u_{N+1}^+},
\end{align*}
both when $z\to0$ through $|\arg z|\leq\pi-\et$ for any small $\et>0$, and hence the proof 
concludes by $N<u_{N+1}^+$. The last $t$-bound here is derived by the following lemma. 
\end{proof}
\begin{lemma}%%%%%%%%%%%%%%%Lemma 6%%%%%%%%%%%%%%%%%%
 For any real $\et$, $\mu$, $\nu$ and $V$ with $\et>0$ and $V\geq0$, 
we have 
\begin{align*}
\int_0^{\infty}e^{-\et v}(v+1)^{\mu}(V+v+1)^{\nu}dv\asymp (V+1)^{\nu},
\tag{5.9}
\end{align*}  
where the $\asymp$-constants depend at most on $\et$, $\mu$ and $\nu$. 
\end{lemma} 
%%%%%%%%%%%%%%%%%%%%%%%%%%%%%%%%%%%%%%%%%%%%%%%%%%%%%%%%%%%%%%%%%%%%%%%%%%%%%%%%%%%%
\begin{proof}
Let the left side denote $I$, and split the integral as
\begin{align*}
I=\int_0^{V+1}+\int_{V+1}^{\infty}=I_1+I_2,
\end{align*}
say. The integration by parts applied to each $I_j$ shows that 
\begin{align*}
I_1\asymp (V+1)^{\nu}(1+e^{-\et V})
\qquad
\text{and}
\qquad
I_2\asymp (V+1)^{\mu+\nu}e^{-\et V},
\end{align*} 
which readily imply the assertion.
\end{proof}

%%%%%%%%%%%%%%%%%%%%%%%%%%%%%%%%%%%%%%%%%%%%%%%%%%%%%%%%%%%%%%%%%%%%%%%%%%%%%%%%%%%%
\begin{proof}[Proof of Theorem~2]%%Proof of Theorem 2
The proof of Theorem~2 is quite similar to that of Theorem~1, instead of moving the path 
in (5.4) to the left. Let $N$ be any nonnegative integer, and choose a constant $u_N^-$ such 
as $-\re\al-N<u_N^-<\min(-\re\al-N+1,0)$. We can then move the path of integration in 
(5.4) from $(u)$ to $(u_{N}^-)$, passing over the poles at $w=-\al-n$ $(n=0,1,\ldots,
N-1)$ of the integrand, and collect its residues to obtain the expression (2.5) with 
\begin{align*}
R_{m,N}^{1,-}(s,a,.\la;z)
=\fr{(-1)^m}{2\pi i}\int_{(u_N^-)}\vG{\al+w,-w}{\al}\ph^{\ast}_{-w-m}(s,a,\la)z^wdw.
\tag{5.10}
\end{align*}

The remaining estimate (2.6) follows similarly to the derivation of (2.4), by moving further 
the path from $(u_N^-)$ to $(u_{N+1}^-)$; this gives 
\begin{align*}
R_{m,N}^{1,-}(s,a,\la;z)
&=\fr{(-1)^N(\al)_N}{N!}\ph^{\ast}_{\al+N-m}(s,a,\la)z^{-\al-N}
+R_{m,N+1}^{1,-}(s,a,\la;z).
\end{align*}  
Here on the right side, the first term is estimated as $\ll(|t|+1)^{\max(0,\lf2-\si\rf)}|z|^{-\re\al-N}$ by (4.12), while the second as $\ll(|t|+1)^{\max(0,\lf2-\si\rf)}|z|^{u_{N+1}^-}$ by applying Lemma~5 with the estimate (5.6); the former bound exceeds the latter when $z\to\infty$ through $|\arg z|\leq\pi-\et$ for any small $\et>0$, since $u_{N+1}^-<-\re\al-N$. The proof of Theorem~2 thus complete. 
\end{proof}
%%%%%Section~6%%%%%%%%%%%%%%%%%%%%%%%%%%%%%%%%%%%%%%%%%%%%%%%%%%%%%%%%%%%%%%%%%%%%%%%%%%%%%%%
\section{Asymptotic expansions for the Riemann-Liouville transforms}%%%%Section 6%%%%%%%%%%
%%%%%%%%%%%%%%%%%%%%%%%%%%%%%%%%%%%%%%%%%%%%%%%%%%%%%%%%%%%%%%%%%%%%%%%%%%%%%%%%%%%%%%
We shall prove Theorems~3~and~4 in this section. Prior to the proofs, several notations are 
prepared in what follows.   

Let $\t1F1{\ka}{\nu}{Z}$ denote Kummer's confluent hypergeometric function of the first kind, 
defined by
\begin{align*}
\1F1{\ka}{\nu}{Z}
=\sum_{k=0}^{\infty}\fr{(\ka)_k}{(\nu)_kk!}Z^k
\qquad(|Z|<+\infty)
\end{align*}
for any complex $\ka$ and $\nu$ with $\nu\notin\mathbb{Z}_{\leq0}$ (cf.~\cite[p.248, 6.1(1)]{erdelyi1953a}), which allows the Euler type integral formula 
\begin{align*}
\1F1{\ka}{\nu}{Z}
=\vG{\nu}{\ka,\nu-\ka}\int_0^1e^{Z\xi}\xi^{\ka-1}(1-\xi)^{\nu-\ka-1}d\xi
\tag{6.1}
\end{align*} 
in $|Z|<+\infty$ and for any complex $\ka$ and $\nu$ with $\re\nu>\re\ka>0$ (cf.~\cite[p.255, 6.5(1)]{erdelyi1953a}).  

Suppose now temporarily that $\si>1$ and $|\arg z|<\pi/2$. We then operate $\RL_{z;\ta}^{\al,\be}$ (see (1.3)) term-by-term on the right sides of (4.13), upon noting (4.17) and using (6.1), that 
\begin{align*}
\begin{split}
\RL_{z;\ta}^{\al,\be}\ph^{(m)}(s+\ta,a,\la)
&=(-1)^m\sum_{l=0}^{\infty}e(\la l)(a+l)^{-s}\log^m(a+l)\\
&\quad\times\1F1{\al}{\al+\be}{-z\log(a+l)},\\
\RL_{z;\ta}^{\al,\be}\ps^{(m)}(s+\ta,a)
&=(-1)^m\int_a^{\infty}\xi^{-s}\log^m\xi\cdot\1F1{\al}{\al+\be}{-z\log\xi}d\xi,
\end{split}
\tag{6.2}
\end{align*}
where the resulting expressions both converge absolutely, since $\t1F1{\al}{\al+\be}{-Z}\asymp Z^{-\al}$ as $Z\to\infty$ through $|\arg Z|<\pi/2$ (cf.~\cite[p.278, 6.13.1(3)]{erdelyi1953a}). 
The Mellin-Barnes formula 
\begin{align*}
\1F1{\ka}{\nu}{Z}=\fr{1}{2\pi i}\int_{(u)}\vG{\ka+w,\nu,-w}{\ka,\nu+w}(-Z)^wdw
\tag{6.3}
\end{align*} 
holds for $|\arg Z|<\pi/2$ with a constant satisfying $-\re\ka<u<0$; this is incorporated in each term on the right sides of (6.2), and then the order of the $w$-integral and the $l$-sum or $\xi$-integral is interchanged to asserts, in view of (1.5), the following lemma.
%%%%%%%%%%%%%%%%%%%%%%%Lemma~6%%%%%%%%%%%%%%%%%%%%%%%%%%%%5
\begin{lemma}%%The beginning of Lemma 6
The formula
\begin{align*}
&\RL_{z;\ta}^{\al,\be}(\ph^{\ast})^{(m)}(s+\ta,a,\la)
\tag{6.4}\\
&\quad=\fr{(-1)^m}{2\pi i}\int_{(u_0)}\vG{\al+w,\al+\be,-w}{\al,\al+\be+w}
\ph^{\ast}_{-w-m}(s,a,\la)z^wdw
\end{align*}  
holds for $|\arg z|<\pi/2$ with a constant $u_0$ satisfying $-\re\al<u_0<0$, where the integral on 
the right side converges absolutely for all $(s,z)$ with $s\in\mathbb{C}$ and 
$|\arg z|<\pi/2${\rm;} this provides the analytic continuation of the left side to the same region 
of $(s,z)$. 
\end{lemma}
\bg{proof}
The integrand in (6.4) on the line $\re w=u$ is (apart from its poles) bounded above, by (4.12) and (5.5), 
as 
\begin{align*}
\ll e^{-|v|(\pi/2-|\arg z|)}(|v|+1)^{-\re\be-u-1/2}(|v|+|t|+1)^{\max(0,\lf2-\si\rf)}|z|^u,
\tag{6.5}
\end{align*}
which asserts the required absolute convergence. 
\ed{proof}
%%%%%%%%%%%%%%%%%%%%%%%%%%%%%%%%%%%%%%%%%%%%%%%%%%%%%%%%%%%%%%%%%%%%%%
We are now ready to prove Theorems 3 and 4.
%%%%%%%%%%%%%%%%%%%%%%%%%%%%%%%%%%%%%%%%%%%%%%%%%%%%%%%%%%%%%%%%%%%%%
\begin{proof}[Proof of Theorem 3]%%Proof of Theorem 3
The same path moving (to the right) argument as in the proof of Theorem~1 leads us 
to the expression (2.8) with 
\begin{align*}
R^{2,+}_{m,N}(s,a,\la;z)
\tag{6.6}
&=\fr{(-1)^m}{2\pi i}\int_{(u_N^+)}\vG{\al+w,\al+\be,-w}{\al,\al+\be+w}
\ph^{\ast}_{-w-m}(s,a,\la)z^wdw
\end{align*}
with a constant $u_N^+$ satisfying $\max(-\re\al,N-1)<u_N^+<N$. The remaining inequality (2.9)  follows similarly to the derivation of (2.4), by using (4.12) and applying Lemma~5 with (6.5); 
the proof of Theorem~3 is complete.
\end{proof}
%%Proof of Theorem~4%%%%%%%%%%%%%%%%%%%%%%%%%%%%
\begin{proof}[Proof of Thereom~4]%%Proof of Theorem 4%%%%%%%%%%%%%%%%5
The connection formula  
\begin{align*}
\1F1{\ka}{\nu}{Z}
&=\vG{\nu}{\nu-\ka}e^{\vep(Z)\pi i\ka}U(\ka;\nu;Z)
+\vG{\nu}{\ka}e^{\vep(Z)\pi i(\ka-\nu)}
\tag{6.7}\\
&\quad\times e^ZU(\nu-\ka;\nu;e^{-\vep(Z)\pi i}Z) 
\end{align*}
holds in the sectors $0<|\arg Z|<\pi$ and for any complex $\ka$ and $\nu$ with $\nu\notin\mathbb{Z}_{\leq0}$ (cf.~\cite[p.259, 6.7.(7)]{erdelyi1953a}\cite[(10.5)]{katsurada-noda2017}); this is incorporated 
in each term on the right sides of (6.2) to yield that 
\begin{align*}
\begin{split}
\RL_{z;\ta}^{\al,\be}\ph^{(m)}(s+\ta,a,\la)
&=(-1)^m\sum_{l=0}^{\infty}e(\la l)(a+l)^{-s}\log^m(a+l)\\
&\quad\times\biggl\{\vG{\al+\be}{\be}e^{-\vep(z)\pi i\al}
U(\al;\al+\be;e^{-\vep(z)\pi i}z\log(a+l))\\
&\quad+\vG{\al+\be}{\al}e^{\vep(z)\pi i\be}(a+l)^{-z}
U(\be;\al+\be;z\log(a+l))\biggr\},\\
%%%%%%%%%%%%%%%%%%%%%%%%%%%%%%%%%%%%%%%%%%%%%%
\RL_{z;\ta}^{\al,\be}\ps^{(m)}(s+\ta,a)
&=(-1)^m\int_a^{\infty}\xi^{-s}\log^m\xi\\
&\quad\times\biggl\{\vG{\al+\be}{\be}e^{-\vep(z)\pi i\al}
U(\al;\al+\be;e^{-\vep(z)\pi i}z\log\xi)\\
&\quad+\vG{\al+\be}{\al}e^{\vep(z)\pi i\be}\xi^{-z}
U(\be;\al+\be;z\log\xi)\biggr\}d\xi
\end{split}
\tag{6.8}
\end{align*}
for any $m\in\mathbb{Z}$ and in the sectors $0<|\arg z|<\pi$. The right sides of (6.8) is further transformed by the Mellin-Barnes formula
\begin{align*}
U(\ka;\nu;Z)
&=\fr{1}{2\pi i}\int_{\cC}\vG{\ka+w,-w,1-\nu-w}{\ka,\ka-\nu+1}Z^wdw
\tag{6.9}
\end{align*}
for $|\arg Z|<3\pi/2$, where $\cC$ is the vertical path which is directed upwards, and suitably  
indented to separate the poles of the integrand at $w=-\ka-m$ $(m=0,1,\ldots)$ from those at 
$w=n$ and $w=1-\nu+n$ $(n=0,1,\ldots)$ (cf.~\cite[p.37, 3.1.2.(3.1.17)]{slater1960}); this is incorporated in each term on the right sides of (6.8), and then the order of the $w$-integral and the $l$-sum or $\xi$-integral is interchanged to assert, in view of (1.5), the following lemma.
\begin{lemma}
For any $m\in\bbZ$ the formula 
\begin{align*}
&\RL_{z;\ta}^{\al,\be}
(\ph^{\ast})^{(m)}(s+\ta,a,\la)
\tag{6.10}\\
&\quad=(-1)^m\vG{\al+\be}{\be}e^{-\vep(z)\pi i\al}I_1(s;z)
+(-1)^m\vG{\al+\be}{\al}e^{\vep(z)\pi i\be}I_2(s;z)
\end{align*}
holds with
\begin{align*}
\begin{split}
I_1(s;z)
&=\fr{1}{2\pi i}
\int_{\cC_1}\vG{\al+w,-w,1-\al-\be-w}{\al,1-\be}
\ph^{\ast}_{-w-m}(s,a,\la)(e^{-\vep(z)\pi i}z)^wdw,\\
%%%%%%%%%%%%%%%%%%%%%%%%%%%%%%%%%%%%%%%%%%%%%%%%%%%%%%%%%%%%%%
I_2(s;z)
&=\fr{1}{2\pi i}
\int_{\cC_2}\vG{\be+w,-w,1-\al-\be-w}{\be,1-\al}\ph^{\ast}_{-w-m}(s+z,a,\la)z^wdw
\end{split}
\tag{6.11}
\end{align*}
in the sectors $0<|\arg z|<\pi$, where the paths $\cC_j$ $(j=1,2)$ are suitably indented to 
separate the poles of the integrand at $w=-\al-m$ $(m=0,1,\ldots)$ if $j=1$, and at $w=-\be-m$ 
$(m=0,1,\ldots)$ if $j=2$ respectively, from those at $w=n$ and $w=1-\al-\be-n$ $(n=0,1,\ldots)$. 
Here the integrals on the right sides of (6.11) converge absolutely for all $(s,z)$ with 
$s\in\mathbb{C}$ and $0<|\arg z|<\pi${\rm;} this provides the analytic continuations of the 
left side of (6.10) to the same regions of $(s,z)$. 
\end{lemma}
\begin{proof}
The integrands on the line $\re w=u$ in (6.11) are (apart from the poles) bounded respectively for $j=1,2$ as, by (4.12) and (5.5), 
\begin{align*}
\begin{split}
&\ll e^{-|v|\{3\pi/2-|\arg z-\vep(z)\pi|\}}(|v|+1)^{-\re\be-u-1/2}(|v|+|t|+1)^{\max(0,\lf2-\si\rf)}|z|^u,
\\
&\ll e^{-|v|(3\pi/2-|\arg z|)}(|v|+1)^{-\re\al-u-1/2}(|v|+|t+y|+1)^{\max(0,\lf2-\si-x\rf)}|z|^u,
\end{split}
\tag{6.12}
\end{align*}
which give the required absolute convergence.
\end{proof}
We can now move upon (6.12) the paths $\cC_j$ $(j=1,2)$ of the integrations in (6.11) both to the right, passing over the poles at $w=-\al-n$ $(n=0,1,\ldots,N_1-1)$ for $I_1(s;z)$,  and at $w=-\be-n$ $(n=0,1,\ldots,N_2-1)$ for $I_2(s;z)$ respectively. If $N_1\geq\lf\re\be\rf$ then $-\re\al-N_1<1-\re(\al+\be)$, while if $N_2\geq\lf\re\al\rf$ then $-\re\be-N_2<1-\re(\al+\be)$; this therefore shows that $\cC_j$ $(j=1,2)$ can be taken under these situations as the vertical straight paths $(u_{j,N_j}^-)$ $(j=1,2)$ respectively with 
\begin{align*}
\min(-\re\al-N_1,1-\re(\al+\be)&<u_{1,N_1}^-<-\re\al-N_1+1,\\
\min(-\re\be-N_2,1-\re(\al+\be))&<u_{2,N_2}^-<-\re\be-N_2+1. 
\end{align*}
We thus obtain for any $N_1\geq\lf\re\be\rf$ and $N_2\geq\lf\re\al\rf$ the expression (2.10) with
\begin{align*}
\begin{split}
R_{1,N_1}^{2,-}(s;z)
&=\fr{1}{2\pi i}
\int_{(u_{1,N_1}^-)}\vG{\al+w,-w,1-\al-\be-w}{\al,1-\be}\\
&\quad\times\ph^{\ast}_{-w-m}(s,a,\la)(e^{-\vep(z)\pi i}z)^wdw,\\
R_{2,N_2}^{2,-}(s;z)
&=\fr{1}{2\pi i}
\int_{(u_{2,N_2}^-)}\vG{\be+w,-w,1-\al-\be-w}{\be,1-\al}\\
&\quad\times\ph^{\ast}_{-w-m}(s+z,a,\la)z^wdw.
\end{split}
\tag{6.13}
\end{align*}

The remaining inequalities (2.11) and (2.12) follow similarly to the derivation of (2.6), where the integrals in (6.13) are estimated by using (4.12) and by applying Lemma~5 with (6.12) respectively; 
the proof of Theorem 4 is thus complete.
\end{proof}    
%%%Section~7%%%%%%%%%%%%%%%%%%%%%%%%%%%%%%%%%%%%%%%%%%%%%%%%%%%%%%%%%%%%%%%%%%%%%%%%%%
\section{Asymptotic expansions for the iterated transforms}%%Section~7
%%%%%%%%%%%%%%%%%%%%%%%%%%%%%%%%%%%%%%%%%%%%%%%%%%%%%%%%%%%%%%%%%%%%%%%%%%%%%%%%%%%%%
We shall prove Theorems~5--10 in this section. The following lemma is prepared for the proofs of Theorems~5~and~6. 
\begin{lemma}
For any real $c>0$, we have 
\begin{align*}
\LM_{z;\ta_2}^{\be}\LM_{\ta_2;\ta_1}^{\al}
e^{-c\ta_1}
&=\fr{1}{2\pi i}\int_{(u_0)}\vG{\al+w,\be+w,-w}{\al,\be}(cz)^wdw
\tag{7.1}
\end{align*}
for $|\arg z|<3\pi/2$ with a constant $u_0$ satisfying $\max(-\re\al,-\re\be)<u_0<0$. 
\end{lemma}
\begin{proof}
Suppose temprarily that $|\arg z|<\pi/2$. On the right side of (3.2) for $f(z)=e^{-cz}$, we use 
(3.1) and 
\begin{align*}
e^{-cz}=\fr{1}{2\pi i}\int_{(u_0)}\vGa(-w)(cz)^wdw\tag{7.2}
\end{align*}
with $\max(-\re\al,-\re\be)<u_0<0$, to find that the left side of (7.1) equals 
\begin{align*}
&\fr{1}{\vGa(\al)\vGa(\be)}\cdot\fr{1}{2\pi i}
\int_{(u)}\vGa(-w)(cz)^w\int_0^{\infty}\xi^{\be-\al-1}
\int_0^{\infty}\ta^{(\al+\be)/2+w-1}e^{-\sqrt{\ta}(\xi+1/\xi)}d\ta d\xi dw\\
&\quad=\fr{1}{\pi i}\int_{(u_0)}\vG{\al+\be+2w,-w}{\al,\be}(cz)^w
\int_0^{\infty}\xi^{2\be+2w-1}(1+\xi^2)^{-\al-\be-2w}d\xi dw, 
\end{align*}  
where the interchange of the order of the $\ta$- and $\xi$-integral is justified (by absolute convergence) with the choice of $u_0$ above, and the resulting inner  $\ta$-integral on the 
first line equals $2\vGa(\al+\be+2w)(\xi+1/\xi)^{-\al-\be-2w}$ for $\xi>0$. 
Furthermore, the last inner $\xi$-integral evaluated as 
$(1/2)\tvG{\al+w,\be+w}{\al+\be+2w}$, by changing the variable $\xi\mapsto\sqrt{(1-\xi)/\xi}$,  
again with the choice of $u_0$ above; this concludes (7.1). Here the temporary 
restriction on $z$ can be relaxed to $|\arg z|<3\pi/2$, since the integrand in (7.1) is of order 
$O\{e^{-|v|(3\pi/2-|\arg z|)}|v|^{\re(\al+\be)+2u_0-1}\}$ as $v\to\pm\infty$.
\end{proof}
Suppose temporarily that $\si>1$ and $|\arg z|<\pi/2$. We operate $\LM_{z;\ta_2}^{\be}\LM_{\ta_2;\ta_1}^{\al}$ term-by-term on the right sides of (4.13) upon (4.17), use (7.1) and then change the 
order of the $w$-integral and the $l$-sum or $\xi$-integral, to obtain, in view of (1.5), the 
following lemma.   
%%%%%%%Lemma~9%%%%%%%%%%%%%%%%%%%%%%5
\begin{lemma}%%%%%%%%%%%%%%%%%%Lemma 9%%%%%%%%%%%%%%%%%%%%%%%%%%%%%%%%
The formula  
\begin{align*}
&\LM_{z;\ta_2}^{\be}\LM_{\ta_2;\ta_1}^{\al}(\ph^{\ast})^{(m)}(s+\ta_1,a,\la)
\tag{7.3}\\
&\quad=\fr{(-1)^m}{2\pi i}\int_{(u_0)}\vG{\al+w,\be+w,-w}{\al,\be}
\ph_{-w-m}^{\ast}(s,a,\la)z^wdw
\end{align*}
holds for any $m\in\mathbb{Z}$ with a constant $u_0$ satisfying $\max(-\re\al,-\re\be)<u_0<0$, where the integral on the right side converges absolutely for all $(s,z)$ with $s\in\mathbb{C}$ and 
$|\arg z|<3\pi/2${\rm;} this provides the analytic continuation of the left side to the same region 
of $(s,z)$. 
\end{lemma}
\begin{proof}%%%%%%%%%%%%%%%%%%%%%%%%%%Proof of Lemma 9%%%%%%%%%%%%%%%%%%%%%%%%%%
The integrand in (7.3) on the line $\re w=u$ is (apart from its poles) bounded, by (4.12) and 
(5.5), as 
\begin{align*}
\ll e^{-|v|(3\pi/2-|\arg z|)}(|v|+1)^{\re(\al+\be)+u-1/2}(|v|+|t|+1)^{\max(0,\lf 2-\si\rf)}|z|^u,
\tag{7.4}
\end{align*}
which asserts the required absolute convergence.
\end{proof}
We are now ready to prove Theorems~5~and~6.
%%%%%%%%%%%Proof of Theorem~5%%%%%%%%%%%%%%%%%%%%%%%
\begin{proof}[Proof of Theorem~5]%%Proof of Theorem~5%%%%%%%%%%%%%%5
The same path moving (to the right) argument as in the proof of Theorem~1 leads us to the expression (3.3) with 
\begin{align*}
R_{m,N}^{3,+}(s,a,\la;z)
&=\fr{(-1)^m}{2\pi i}\int_{(u_N^+)}\vG{\al+w,\be+w,-w}{\al,\be}
\ph_{-w-m}^{\ast}(s,a,\la)z^wdw
\tag{7.5}
\end{align*}
with $\max(-\re\al,N-1)<u_N^+<N$. The remaining inequality (3.4) follows similarly to the derivation 
of (2.9), by using (4.12) and applying Lemma~5 with (7.4). 
\end{proof}
%%%%%%%%%%%%%%%%%%%%%%%%%%%%%%%%%%%%%%%%%%%%%%%%%%%%%%%%%%%%%%%%%%%%%%%%%%%%%%%%%%
\begin{proof}[Proof of Thoerem~6]
For any nonnegative integer $N$, let a real constant $u_N^-$ and an integer $N'$ satisfy 
\begin{align*}
-\re\al-N<u_N^-<-\re\be-N'\leq-\re\al-N+1,
\tag{7.6}
\end{align*} 
which gives $N'=N-\lf\re(\be-\al)\rf$. We can then move the path in (7.3) from $(u_0)$ to $(u_N^-)$, passing over the poles of the integrand at $w=-\al-n$ $(n=0,1,\ldots,N-1)$ and $w=-\be-n$ $(n=0,1,\ldots,N'-1)$, to obtain the expression (3.5) with 
\begin{align*}
R_{m,N}^{3,-}(s,a,\la;z)
&=\fr{(-1)^m}{2\pi i}\int_{(u_N^-)}\vG{\al+w,\be+w,-w}{\al,\be}
\ph^{\ast}_{-w-m}(s,a,\la)z^wdw
\tag{7.7}
\end{align*} 
for $|\arg z|<3\pi/2$, where each term of the asymptotic series is rewritten by 
\begin{align*}
\vGa(s-n)=(-1)^n\vGa(s)/(1-s)_n
\qquad(n\in\mathbb{Z}).
\tag{7.8}
\end{align*}

The remaining inequality (3.6) follows by moving further the path in (7.5) from $(u_N^-)$ 
to $(u_{N+1}^-)$; this gives
\begin{align*}
R_{m,N}^{3,-}(s,a,\la;z)
&=\fr{(-1)^N(\al)_N(\be)_N}{N!}\ph^{\ast}_{\al+N-m}(s,a,\la)z^{-\al-N}\\
&\quad+\fr{(-1)^{N'}(\al)_{N'}(\be)_{N'}}{N'!}\ph^{\ast}_{\be+N'-m}(s,a,\la)z^{-\be-N'}
+R_{m,N+1}^{3,-}(s,a,\la;z),
\end{align*}
where the first and second terms on the right side are estimated respectively as 
\begin{align*}
\ll(|t|+1)^{\max(0,\lf2-\si\rf)}|z|^{-\re\al-N}
\quad
\text{and}
\quad
\ll(|t|+1)^{\max(0,\lf2-\si)}|z|^{-\re\be-N'}
\end{align*}
by (4.12), while the third as $\ll(|t|+1)^{\max(0,\lf2-\si\rf)}|z|^{u_{N+1}^-}$ by applying Lemma~5 
with (7.4). The first term, in view of (7.6), exceeds other terms when $z\to\infty$ through $|\arg z|\leq3\pi/2-\et$ for any small $\et>0$; this concludes (3.6).
\end{proof}
%%%%%%%%%%%%%%%%%%%%%%%%%%%%%%%%%%%%%%%%%%%%%%%%%%%%%%%%%%%%%%%%%%%%%%%%%%%%%

We next proceed to prove Theorems~7~and~8. Prior to the proofs, the following lemma is prepared. 
\begin{lemma}
For any real $c>0$, we have 
\begin{align*}
\RL_{z;\ta_2}^{\be,\ga}\LM_{\ta_2;\ta_1}^{\al}e^{-c\ta_1}
&=\fr{1}{2\pi i}\int_{(u_0)}\vG{\al+w,\be+w,\be+\ga,-w}{\al,\be,\be+\ga+w}(cz)^wdw
\tag{7.9}
\end{align*} 
for $|\arg z|<\pi$ with a constant $u_0$ satisfying $\max(-\re\al,-\re\be)<u_0<0$. 
\end{lemma}
%%%%%%%%%%%%%Proof of Lemma~9%%%%%%%%%%%%%%%%%%%%%%%%%%%%%%%%%%
\begin{proof}%%Proof of Lemma~9%%%%%%%%%%%%%%%%%%%%%%%55
Suppose temporarily that $|\arg z|<\pi/2$. Then the left side of (3.9) for $f(z)=e^{-cz}$ 
equals, by using (3.8) on the right side of (3.9),  
\begin{align*}
&\vG{\be+\ga}{\al,\be,\ga}\int_0^{\infty}e^{-(1+cz)\ta}\ta^{\al-1}
\int_0^{\infty}e^{-\ta\xi}\xi^{\ga-1}(1+\xi)^{\al-\be-\ga}d\xi d\ta\\
&\quad=\vG{\be+\ga}{\be,\ga}
\int_0^{\infty}\xi^{\ga-1}(1+\xi)^{\al-\be-\ga}(1+\xi+cz)^{-\al}d\xi
\end{align*}
where the interchange the order of the $\xi$- and $\ta$-integral on the first line is justified 
by absolute convergence; the resulting inner $\ta$-integral is evaluated as 
$\vGa(\al)(1+\xi+cz)^{-\al}$. Here the last $\xi$-integral is transformed by changing the 
variable $\xi\mapsto(1-\xi)/\xi$ to show upon (3.17) that 
\begin{align*}
\RL_{z;\ta_2}^{\be,\ga}\LM_{\ta_2;\ta_1}^{\al}e^{-c\ta_1}
&=\F{\al,\be}{\be+\ga}{-cz},
\tag{7.10}
\end{align*}
where the temporary restriction on $z$ can be relaxed to $|\arg z|<\pi$ by analytic continuation. 
The Mellin-Barnes formula 
\begin{align*}
\F{\ka,\mu}{\nu}{Z}
&=\fr{1}{2\pi i}\int_{(u)}\vG{\ka+w,\mu+w,\nu,-w}{\ka,\mu,\nu+w}(-Z)^wdw
\tag{7.11}
\end{align*}
for $|\arg(-Z)|<\pi$ with $\max(-\re\ka,-\re\mu)<u<0$ (cf.~\cite[p.62, 2.1.3(15)]{erdelyi1953a}) therefore concludes (7.9). 
\end{proof}

Suppose temporarily that $\si>1$ and $|\arg z|<\pi/2$. We operate 
$\RL_{z;\ta_2}^{\be,\ga}\LM_{\ta_2;\ta_1}^{\al}$ term-by-term on the right sides of (4.13), note (4.17), and then use (3.17) in each resulting term, to find for any $m\in\mathbb{Z}$ that 
\begin{align*}
\begin{split}
\RL_{z;\ta_2}^{\be,\ga}\LM_{\ta_2;\ta_1}^{\al}\ph^{(m)}(s+\ta_1,a,\la)
&=(-1)^m\sum_{l=0}^{\infty}e(\la l)(a+l)^{-s}\log^m(a+l)\\
&\quad\times\F{\al,\be}{\be+\ga}{-z\log(a+l)},\\
\RL_{z;\ta_2}^{\be,\ga}\LM_{\ta_2;\ta_1}^{\al}\ps^{(m)}(s+\ta_1,a)
&=(-1)^m\int_a^{\infty}\xi^{-s}\log^m\xi\cdot\F{\al,\be}{\be+\ga}{-z\log\xi}d\xi,
\end{split}
\tag{7.12}
\end{align*} 
in which (7.11) is incorporated to assert, in view of (1.5), the following lemma.
\begin{lemma}
The formula
\begin{align*}
&\RL_{z;\ta_2}^{\be,\ga}\LM_{\ta_2;\ta_1}^{\al}
(\ph^{\ast})^{(m)}(s+\ta_1,a,\la)\tag{7.13}\\
&\quad=\fr{(-1)^m}{2\pi i}
\int_{(u_0)}\vG{\al+w,\be+w,\be+\ga,-w}{\al,\be,\be+\ga+w}\ph^{\ast}_{-w-m}(s,a,\la)z^wdw
\end{align*}  
holds for any $m\in\mathbb{Z}$ with a constant $u_0$ satisfying $\max(-\re\al,-\re\be)<u_0<0$, where the integral on the right side converges absolutely for all $(s,z)$ with $s\in\mathbb{C}$ and $|\arg z|<\pi${\rm;} this provides the analytic continuation of the left side to the same region of $(s,z)$.
\end{lemma}
\begin{proof}
The integrand in (7.13) on the line $\re w=u$  is (apart from its poles) bounded, by (4.12) and (5.5), as 
\begin{align*}
\ll e^{-|v|(\pi-|\arg z|)}(|v|+1)^{\re(\al-\ga)-1}(|v|+|t|+1)^{\max(0,\lf2-\si\rf)}|z|^u,
\tag{7.14}
\end{align*}
which gives the required absolute convergence.
\end{proof}

We are now ready to prove Theorems~7~and~8.
%%%%%%%%%%%%%%%%Proof of Theorem~7%%%%%%%%%%%%%%%%%%%%%
\begin{proof}[Proof of Theorem~7]%%Proof of Theorem~7%%%%%%%%%%%%%%%%%%%%%%%%%%%%%%5
The same path moving (to the right) argument as in the proof of Theorem~1 leads us to the expression (3.10) with
\begin{align*}
R_{m,N}^{4,+}(s,a,\la;z)
&=\fr{(-1)^m}{2\pi i}\int_{(u_N^+)}\vG{\al+w,\be+w,\be+\ga,-w}{\al,\be,\be+\ga+w}
\tag{7.15}\\
&\quad\times\ph_{-w-m}^{\ast}(s,a,\la)z^wdw
\end{align*} 
for $|\arg z|<\pi$ with $\max(-\re\al,N-1)<u_N^+<N$. The remaining estimate follows 
similarly to the derivation of (2.9) by using (4.12) and applying Lemma 5 with (7.14). 
\end{proof}
\begin{proof}[Proof of Theorem~8]
The connection formula
\begin{align*}
\F{\ka,\mu}{\nu}{Z}
&=\vG{\nu,\mu-\ka}{\mu,\nu-\ka}(-Z)^{-\al}\F{\ka,1-\nu+\ka}{1-\mu+\ka}{\fr{1}{Z}}
\tag{7.16}\\
&\quad+\vG{\nu,\ka-\mu}{\ka,\nu-\mu}(-Z)^{-\be}\F{\mu,1-\nu+\mu}{1-\ka+\mu}{\fr{1}{Z}}
\end{align*}
holds for $|\arg(-Z)|<\pi$ and any complex $\ka$, $\mu$ and $\nu$ with $\nu\notin\mathbb{Z}_{\leq0}$ and $\ka-\mu\notin\mathbb{Z}$ (cf.~\cite[p.108, 2.10(2)]{erdelyi1953a}); this is used in each term on the right sides of (7.12) to assert that
\begin{align*}
\begin{split}
&\RL_{z;\ta_2}^{\be,\ga}\LM_{\ta_2;\ta_1}^{\al}\ph^{(m)}(s+\ta_1,a,\la)\\
&\quad=(-1)^m\sum_{l=0}^{\infty}e(\la l)(a+l)^{-s}\log^m(a+l)\\
&\qquad\times\biggl[\vG{\be+\ga,\be-\al}{\be,\be+\ga-\al}\{z\log(a+l)\}^{-\al}
\F{\al,1-\be-\ga+\al}{1-\be+\al}{\fr{1}{z\log(a+l)}}\\
&\qquad+\vG{\be+\ga,\al-\be}{\al,\ga}\{z\log(a+l)\}^{-\be}
\F{\be,1-\ga}{1-\al+\be}{\fr{1}{z\log(a+l)}}\biggr],\\
&\RL_{z;\ta_2}^{\be,\ga}\LM_{\ta_2;\ta_1}^{\al}\ps^{(m)}(s+\ta_1,a)\\
&\quad=(-1)^m\int_a^{\infty}\xi^{-s}\log^m\xi\\
&\qquad\times\biggl\{\vG{\be+\ga,\be-\al}{\be,\be+\ga-\al}(z\log\xi)^{-\al}
\F{\al,1-\be-\ga+\al}{1-\be+\al}{\fr{1}{z\log\xi}}\\
&\qquad+\vG{\be+\ga,\al-\be}{\al,\ga}(z\log\xi)^{-\be}
\F{\be,1-\ga}{1-\al+\be}{\fr{1}{z\log\xi}}\biggr\}d\xi
\end{split}
\tag{7.17}
\end{align*}
for any $m\in\mathbb{Z}$ and $|\arg z|<\pi$. 
\end{proof}
We incorporate (7.11) in each term on the right side of (7.17), and then chang the order of the 
$w$-integral and the $l$-sum or $\xi$-integral, we obtain, in view of (1.5), the following lemma.
%%%%%%%%Lemma~12%%%%%%%%%%%%%%%%%%%%%%
\begin{lemma}%%Lemma~12%%%%%%%%%%%%%%%%%%%%%%
The formula
\begin{align*}
&\RL_{z;\ta_2}^{\be,\ga}\LM_{\ta_2;\ta_1}^{\al}(\ph^{\ast})^{(m)}(s+\ta_1,a,\la)
\tag{7.18}\\
&\quad=(-1)^m\vG{\be+\ga,\be-\al}{\be,\be+\ga-\al}J_1(s;z)
+(-1)^m\vG{\be+\ga,\al-\be}{\al,\ga}J_2(s;z)
\end{align*}
holds with
\begin{align*}
\begin{split}
J_1(s;z)
&=\fr{1}{2\pi i}\int_{\cC_1}
\vG{\al+w,1-\be-\ga+\al+w,1-\be+\al,-w}{\al,1-\be-\ga+\al,1-\be+\al+w}\\
&\quad\times\ph^{\ast}_{\al+w-m}(s,a,\la)z^{-\al-w}dw,\\
J_2(s;z)
&=\fr{1}{2\pi i}\int_{\cC_2}
\vG{\be+w,1-\ga+w,1-\al+\be,-w}{\be,1-\ga,1-\al+\be+w}\\
&\quad\times\ph^{\ast}_{\be+w-m}(s,a,\la)z^{-\be-w}dw
\end{split}
\tag{7.19}
\end{align*}
for any $m\in\mathbb{Z}$, where the paths $\cC_j$ $(j=1,2)$ are direcetd upwards and suitably indented to separate the poles of the integrand at $w=n$ $(n=0,1,\ldots)$ from those at $w=-\al-n$ and $w=-\al+\be+\ga-1-n$ $(n=0,1,\ldots)$ for $J_1(s;z)$, while at $w=n$ $(n=0,1,\ldots)$ from those at $w=-\be-n$ and $w=\ga-1-n$ $(n=0,1,\ldots)$ for $J_2(s;z)$. Here the integrals on the right 
sides of (7.19) converges absolutely for all $(s,z)$ with $s\in\mathbb{C}$ and $|\arg z|<\pi${\rm;} 
this provides the analytic continuations of the left side of (7.18) to the same region of $(s,z)$.
\end{lemma}
%%%%%%%%%%%Proof of Lemma~12%%%%%%%%%%%%%%%%%%%%%%
\begin{proof}%%Proof of Lemma~12%%%%%%%%%%%%%%%%%
The integrands in (7.19) on the line $\re w=u$ are (apart from their poles) bounded respectively for $j=1,2$ as, by (4.12) and (5.5),   
\begin{align*}
\begin{split}
&\ll e^{-|v|(\pi-|\arg z|)}(|v|+1)^{\re\al-1}(|v|+|t|+1)^{\max(0,\lf2-\si\rf)}|z|^{-\re\al-u},\\
&\ll e^{-|v|(\pi-|\arg z|)}(|v|+1)^{\re(\al-\ga)-1}(|v|+|t|+1)^{\max(0,\lf2-\si\rf)}|z|^{-\re\be-u},
\end{split}
\tag{7.20}
\end{align*}
which give the required absolute convergence.
\end{proof}

We now move the paths $\cC_j$ $(j=1,2)$ in (7.19) to the right, passing over the poles of the integrand at $w=n$ $(n=0,1,\ldots,N_1-1)$ for $J_1(s;z)$, while at $w=n$ $(n=0,1,\ldots,N_2-1)$ for $J_2(s;z)$. 
If $N_1\geq\lf\re(\be+\ga-\al)\rf$ then $\re(\be+\ga-\al)-1<N_1$, while if $N_2\geq\lf\re\ga\rf$ then $\re\ga-1<N_2$; this shows that $\cC_j$ $(j=1,2)$ can be taken under these situations as the 
vertical straight lines $(u_{j,N_j})$ respectively for $j=1,2$ with 
\begin{align*}
\begin{split}
\max(N_1-1,\re(\be+\ga-\al))&<u_{1,N_1}<N_1,\\
\max(N_2-1,\re\ga-1)&<u_{2,N_2}<N_2.
\end{split}
\tag{7.21}
\end{align*}   
We thus obtain for any $N_1\geq\lf\re(\be+\ga-\al)\rf$ and $N_2\geq\lf\re\ga-1\rf$ the 
expression (3.12) with 
\begin{align*}
\begin{split}
R_{m,N_1}^{4,-}(s,a,\la;z)
&=\fr{1}{2\pi i}\int_{(u_{1,N_1})}
\vG{\al+w,1-\be-\ga+\al+w,1-\be-\al,-w}{\al,1-\be-\ga+\al,1-\be+\al+w}\\
&\quad\times\ph^{\ast}_{\al+w-m}(s,a,\la)z^{-\al-w}dw,\\
R_{m,N_2}^{4,-}(s,a,\la;z)
&=\fr{1}{2\pi i}\int_{(u_{2,N_2})}
\vG{\be+w,1-\ga+w,1-\al+\be,-w}{\be,1-\ga,1-\al+\be+w}\\
&\quad\times\ph^{\ast}_{\be+w-m}(s,a,\la)z^{-\be-w}dw.
\end{split}
\tag{7.22}
\end{align*}

The remaining inequalities in (3.13) follow similarly to the derivation of (2.11) and (2.12), by using (4.12) and by applying Lemma~5 with (7.20); the proof of Theorem~8 is thus complete.

We finally proceed to prove Theorems~9~and~10. Prior to the proofs the following lemma is prepared. 
%%%%%%%%%%%Lemma~13%%%%%%%%%%%%%%%%%%%%%%%%%%%%%%%%%%%%55
\begin{lemma}%%%%%%%%%%%%%%%%%%%Lemma 13%%%%%%%%%%%%%%%%%%%%%%%
For any real $c>0$, we have  
\begin{align*}
\RL_{z;\ta_2}^{\ga,\de}\RL_{\ta_2;\ta_1}^{\al,\be}e^{-c\ta_1}
=\fr{1}{2\pi i}
\int_{(u)}\vG{\al+w,\ga+w,\al+\be,\ga+\de,-w}{\al,\ga,\al+\be+w,\ga+\de+w}(cz)^wdw
\tag{7.23}
\end{align*} 
for $|\arg z|<\pi/2$ with a constant $u$ satisfying $\max(-\re\al,-\re\ga)<u<0$.
\end{lemma}
\begin{proof}
Suppose temporarily that $|\arg z|<\pi/2$ and  
\begin{align*}
\re\ga<\re(\al+\be).
\tag{7.24} 
\end{align*}
Using (7.2) on the right side of (3.16) for $f(z)=e^{-cz}$, and then changing the order of the 
$\ta$- and $w$-integral, we find that
\begin{align*}
\RL_{z;\ta_2}^{\ga,\de}\RL_{\ta_2;\ta_1}^{\al,\be}
e^{-c\ta_1}
&=\vG{\al+\be,\ga+\de}{\al,\ga,\be+\de}
\cdot\fr{1}{2\pi i}\int_{(u_0)}\vGa(-w)G(\al,\be,\ga,\de;w)(cz)^wdw,
\tag{7.25}
\end{align*}  
say, where 
\begin{align*}
G(\al,\be,\ga,\de;w)
&=\int_0^1\ta^{\al+w-1}(1-\ta)^{\be+\de-1}\F{\al+\be-\ga,\de}{\be+\de}{1-\ta}d\ta,
\tag{7.26}
\end{align*}
and a constant $u_0$ is chosen as 
\begin{align*}
\max(-\re\al,-\re\ga)<u<0. 
\tag{7.27}
\end{align*}
It can be shown just below that  
\begin{align*}
G(\al,\be,\ga,\de;w)
&=\vG{\al+w,\ga+w,\be+\de,}{\al+\be+w,\ga+\de+w},\tag{7.28}
\end{align*}
which is substituted into the integrand in (7.25) to conclude (7.23); the temporary 
restriction (7.23) can be relaxed by analytic continuation. 

We now proceed to prove (7.28). Rewriting the hypergeometric function in (7.26) by 
Kummer's transformation 
\begin{align*}
\F{a,b}{c}{Z}=(1-Z)^{-a}\F{a,c-b}{c}{\fr{Z}{Z-1}}
\end{align*}
(cf.~\cite[p.105, 2.9(3)]{erdelyi1953a}), and then using (7.11) with $-Z=(1-\ta)/\ta$ in the resulting expression, we have 
\begin{align*}
&\int_0^1\ta^{\ga-\be+w-1}(1-\ta)^{\be+\de-1}
\cdot\fr{1}{2\pi i}\int_{(\rh)}\vG{\al+\be-\ga+r,\be+r,\be+\de,-r}{\al+\be-\ga,\be,\be+\de+r}
\Bigl(\fr{1-\ta}{\ta}\Bigr)^rdrd\ta\\
&\quad=\vG{\be+\de}{\al+\be-\ga,\be,\ga+\de+w}
\fr{1}{2\pi i}\int_{(\rh)}\vG{\al+\be-\ga+r, \be+r, -r, \ga-\be+w-r}{1}dr.
\end{align*}
Here real $\rh$ is chosen with $\max(-\re(\al+\be-\ga), -\re\be)<\rh
<\min(0,\re(\ga-\be)+u_0)$, which is possible under (7.24) and (7.27), and the interchange of the 
$r$- and $\ta$-integral is justified (upon Fubini's theorem) by absolute convergence. 
Furthermore, the last $r$-integral is evaluated by a particular case of Barnes' first lemma, which 
states that 
\begin{align*}
\fr{1}{2\pi i}\int_{(\si)}\vG{a+s,b+s,c-s,d-s}{1}ds
&=\vG{a+c,a+d,b+c,b+d}{a+b+c+d},
\end{align*}
for any complex $a$, $b$, $c$, $d$ and real $\si$ satisfying 
$\max(-\re a,-\re b)<\si<\min(\re c,\re d)$ (cf.~\cite[p.50, 1.19(8)]{erdelyi1953a}); 
this completes the proof.
\end{proof}
%%%%%%%%%%%%%%%%%%%%%%%%%%%%%%%%%%%%%%%%%%%%%%%%%%%%%%%%%%%%%%%%%%%%%%%%%%%%%%%%%%%

Suppose temporarily that $\si>1$ and $|\arg z|<\pi/2$. We operate $\RL^{\ga,\de}_{z;\ta_2}\RL^{\al,\be}_{\ta_2;\ta_1}$ term-by-term on the right sides of (4.13), upon (4.17), incorporate (7.23) in each resulting term, and then change the order of the $w$-integral and the $l$-sum or $\xi$-integral, to obtain, in view of (1.5), the following lemma.  
\begin{lemma}
The formula 
\begin{align*}
&\RL_{z;\ta_2}^{\ga,\de}\RL_{\ta_2;\ta_1}^{\al,\be}
(\ph^{\ast})^{(m)}(s+\ta_1,a,\la)\tag{7.29}\\
&\quad=
\fr{(-1)^m}{2\pi i}\int_{(u_0)}
\vG{\al+w,\be+w,\al+\be,\ga+\de,-w}{\al,\ga,\al+\be+w,\ga+\de+w}\ph^{\ast}_{-w-m}(s,a,\la)z^wdw
\end{align*}
holds for any $m\in\mathbb{Z}$ with a constant $u_0$ satisfying $\max(-\re\al,-\re\ga)<u_0<0$, where the integral on the right side converges absolutely for all $(s,z)$ with $s\in\mathbb{C}$ 
and $|\arg z|<\pi/2${\rm;} this provides the analytic continuation of the left side to the same 
region of $(s,z)$.
\end{lemma}
%%%%%%%%%%%%%%%%%%%%%%%%%%%%%%%%%%%%%%%%%%%%%%%%%%%%%%%%%%%%%%%%%%%%%%%%%%%%%%%%%%%%%
\begin{proof}
The integrand in (7.29) on the line $\re w=u$ is (apart from its poles) bounded, by (4.12) and (5.5), 
as 
\begin{align*}
\ll e^{-|v|(\pi/2-|\arg z|)}(|v|+1)^{\re(\be+\de)-u-1/2}(|v|+|t|+1)^{\max(0,\lf2-\si\rf)}|z|^u,
\tag{7.30}
\end{align*}  
which gives the required absolute convergence. 
\end{proof}
%%%%%%%%%%%%%%%%%%%%%%%%%%%%%%%%%%%%%%%%%%%%%%%%%%%%%%%%%%%%%%%%%%%%%%%%%%%%%%%%%%%%

We are now ready to prove Theorems~9~and~10.
\begin{proof}[Proof of Theorem~9]
The same path moving (to the right) argument as in the proof of Theorem~1 leads us to the expression (3.18) with 
\begin{align*}
R_{m,N}^{5,+}(s,a,\la;z)
&=\fr{(-1)^m}{2\pi i}\int_{(u_N^+)}
\vG{\al+w,\ga+w,\al+\be,\ga+\de,-w}{\al,\ga,\al+\be+w,\ga+\de+w}
\tag{7.31}\\
&\quad\times\ph_{-w-m}^{\ast}(s,a,\la)z^wdw
\end{align*}
for $|\arg z|<\pi/2$ with a constant $u_N^+$ satisfying $\max(-\re\al,-\re\ga,N-1)<u_N^+<N$. 
The remaining inequality (3.19) follows simiraly to the derivation of (2.6), where the integral in (7.31) is bounded by using (4.12) and by applying Lemma~5 with (7.30). 
\end{proof}
%%%%%%%%%%%%%%%%
\begin{proof}[Proof of Theorem~10]
For any nonnegative integer $N$, let a real constant $u_N^-$ and an integer $N'$ satisfy 
\begin{align*}
-\re\al-N<u_N^-<-\re\ga-N'+1\leq-\re\al-N+1,
\end{align*}
which gives $N'=N-\lf\re(\ga-\al)\rf$. Then 
the same path moving (to the left) argument as in the proof of Theorem~6 leads us to the 
expression (3.20) with 
\begin{align*}
R_{m,N}^{5,-}(s,a,\la;z)
&=\fr{(-1)^m}{2\pi i}\int_{(u_N^-)}
\vG{\al+w,\ga+w,\al+\be,\ga+\de,-w}{\al,\ga,\al+\be+w,\ga+\de+w}
\tag{7.32}\\
&\quad\times\ph_{-w-m}^{\ast}(s,a,\la)z^wdw
\end{align*} 
for $|\arg z|<\pi/2$, where each term of the asymptotic series is rewritten by (7.8). 
The remaining inequality follows similarly to the derivation of (3.6), where the integral 
in (7.32) is bounded by using (5.12) and applying Lemma~5 with (7.30); 
this completes the proof.
\end{proof}
%%

%%%%%%%%%%%%%%%%%%%%%%%%%%%%%%%%%%%%%%%%%%%%%%%%%%%%%%%%%%%%%%%%%%%%%%%%%%%%%%%%%%%%%%%%%%%%%%%%%%%%%%%%%%%
\end{document}